\documentclass[11 pt]{amsart}
\usepackage[utf8]{inputenc}
\usepackage[english]{babel}
\usepackage{amsmath}
\usepackage{amssymb}
\usepackage{amsfonts}
\usepackage{amsthm}
\usepackage{mathrsfs}
\usepackage{graphicx}
\usepackage{enumitem}
\setlist[itemize]{leftmargin=*}
\usepackage{xcolor}
\usepackage{url}
\usepackage{bm}
\usepackage[all]{xy}
\usepackage{xcolor}

\usepackage{ esint }
\usepackage{xcolor}
\usepackage{soul}
\usepackage{stmaryrd}

\usepackage{doi}
\usepackage{hyperref}
\hypersetup{
	colorlinks,
	linkcolor={blue!60!black},
	citecolor={green!60!black},
	urlcolor={green!60!black}
}
\usepackage[capitalise,nameinlink]{cleveref}
\usepackage{microtype}
\usepackage[maxnames=50,doi=false,isbn=false,url=false,eprint=false,backend=bibtex,giveninits=true,sorting=nyt]{biblatex}
\DeclareFieldFormat[article,book,misc,inbook]{title}{\textbf{#1}}
\DeclareFieldFormat[inbook]{chapter}{#1}
\renewbibmacro{in:}{}
\DeclareFieldFormat[article]{volume}{vol\adddot\addnbspace #1}
\DeclareFieldFormat[article]{number}{no\adddot\addnbspace #1}
\renewbibmacro{volume+number+eid}{%
	\setunit{\addcomma\space}%
	\printfield{volume}%
	\setunit{\addcomma\space}%
	\printfield{number}%
	\setunit{\addcomma\space}%
	\printfield{eid}%
}
\AtEveryBibitem{
  \clearlist{language}
}

\crefformat{equation}{#2(#1)#3}
\crefrangeformat{equation}{#3(#1)#4 to~#5(#2)#6}
\crefmultiformat{equation}{#2(#1)#3}{ and~#2(#1)#3}{, #2(#1)#3}{ and~#2(#1)#3}
\usepackage{mathtools}
\usepackage{etoolbox}
\usepackage{chngcntr}

\newcommand{\R}{\mathbb{R}}
\newcommand{\C}{\mathbb{C}}

\newcommand{\G}{\mathcal{G}}

\newcommand{\de}{\partial}

\newcommand{\weakto}{\rightharpoonup}

\DeclareMathOperator{\diam}{diam}

\newcommand{\hau}{\mathcal{H}}

\newcommand{\vol}{\mathrm{vol}}
\newcommand{\ang}[1]{\langle #1\rangle}

\renewcommand{\epsilon}{\varepsilon}

\theoremstyle{definition}
\newtheorem{definition}{Definition}
\newtheorem{rmk}[definition]{Remark}
\newtheorem*{definition*}{Definition}
\newtheorem*{rmk*}{Remark}
\newtheorem*{ack*}{Acknowledgement}
\newtheorem*{acks*}{Acknowledgements}
\theoremstyle{plain}
\newtheorem{thm}[definition]{Theorem}
\newtheorem{lemma}[definition]{Lemma}
\newtheorem{corollary}[definition]{Corollary}
\newtheorem{proposition}[definition]{Proposition}

\newtheorem*{thm*}{Theorem}

\newtheorem*{lemma*}{Lemma}
\newtheorem*{corollary*}{Corollary}
\newtheorem*{proposition*}{Proposition}
\newtheorem*{claim*}{Claim}
\newtheorem*{conj*}{Conjecture}
\newtheorem*{problem*}{Open problem}

\numberwithin{equation}{section}
\numberwithin{definition}{section}
\usepackage[top=3cm,bottom=3cm,left=3.5cm,right=3.5cm]{geometry}
\renewcommand{\div}{\text{div}}

\makeatletter
\renewcommand{\tocsection}[3]{%
  \indentlabel{\@ifnotempty{#2}{\bfseries\ignorespaces#1 #2\quad}}\bfseries#3}
\renewcommand{\tocsubsection}[3]{%
  \indentlabel{\@ifnotempty{#2}{\ignorespaces#1 #2\quad}}#3}

\newcommand\@dotsep{4.5}
\def\@tocline#1#2#3#4#5#6#7{\relax
  \ifnum #1>\c@tocdepth 
  \else
    \par \addpenalty\@secpenalty\addvspace{#2}%
    \begingroup \hyphenpenalty\@M
    \@ifempty{#4}{%
      \@tempdima\csname r@tocindent\number#1\endcsname\relax
    }{%
      \@tempdima#4\relax
    }%
    \parindent\z@ \leftskip#3\relax \advance\leftskip\@tempdima\relax
    \rightskip\@pnumwidth plus1em \parfillskip-\@pnumwidth
    #5\leavevmode\hskip-\@tempdima{#6}\nobreak
    \leaders\hbox{$\m@th\mkern \@dotsep mu\hbox{.}\mkern \@dotsep mu$}\hfill
    \nobreak
    \hbox to\@pnumwidth{\@tocpagenum{\ifnum#1=1\bfseries\fi#7}}\par
    \nobreak
    \endgroup
  \fi}
\AtBeginDocument{%
\expandafter\renewcommand\csname r@tocindent0\endcsname{0pt}
}
\def\l@subsection{\@tocline{2}{0pt}{2.5pc}{5pc}{}}

\makeatother

\begin{document}
\title[Quantitative stability for abelian Higgs in 2D]{Quantitative stability of Yang-Mills-Higgs instantons in two dimensions}
	\author[A. Halavati]{Aria Halavati}
	\address{Courant Institute of Mathematical Sciences, New York University, 251 Mercer Street, New York, NY 10012, United States of America.}
	\email{aria.halavati@cims.nyu.edu}

	\begin{abstract}
        We prove that if an N-vortex pair nearly minimizes the Yang-Mills-Higgs energy, then it is second order close to a minimizer. First we use new weighted inequalities in two dimensions and compactness arguments to show stability for sections with some regularity. Second we define a selection principle using a penalized functional and by elliptic regularity and smooth perturbation of complex polynomials, we generalize the stability to all nearly minimizing pairs. With the same method, we also prove the analogous second order stability for nearly minimizing pairs on nontrivial line bundles over arbitrary compact smooth surfaces.
	\end{abstract}
	
    \dedicatory{Dedicated to the memory of Nika Shakarami (2005--2022)}

    \maketitle
    \small
    \tableofcontents 
    \normalsize
    \frenchspacing

    \section{Introduction}
    \subsection{Background and main results} Let $(u,\nabla)$ be a section and connection on the trivial line bundle $\R^2\times\C \rightarrow \R^2$. The \textit{self-dual U(1)-Yang-Mills-Higgs} functional after a suitable re-scaling is
    \begin{align}\label{energy}
        E(u,\nabla) = \int_{\R^2}|\nabla u|^2 + |F_\nabla|^2 + \frac{(1-|u|^2)^2}{4}\,,
    \end{align}
   where $F_\nabla$ is the curvature two-form of $\nabla$ (see \cref{section-2} for details). One can show that in two dimensions the energy is lower bounded by a topological constant (see \cite{Bogomolny,Taubes})
    \begin{align}\label{easy-lower-bound}
        E(u,\nabla) \geq 2\pi |N|\,,
    \end{align}
    where $N=\deg(u)$ is the rotation number of $\frac{u}{|u|}$ at infinity (Which is well defined when the energy \cref{energy} is finite, see \cref{IBP}). It is well known that minimizers of this functional satisfy a system of first order \textit{vortex equations}, also known as the Bogomolny equations. In his PhD thesis (\cite{Taubes,Taubes-1}) C.H.Taubes shows that after prescribing the zero set, the solution \textit{exists} and is \textit{unique}, up to a change of gauge (see also \cite{Taubes-2}). Later in \cite{Pigati-1}, A.Pigati and D.Stern consider the $\epsilon$-rescaled Yang-Mills-Higgs energy in higher dimensions:
  \begin{align*}
      E_\epsilon(u,\nabla) = \int_{M^n}|\nabla u|^2 + \epsilon^2|F_\nabla|^2 + \frac{(1-|u|^2)^2}{4\epsilon^2},
  \end{align*}
 where $M^n$ is a compact Rimmannian $n$-manifold and they use this energy to construct minimal sub-manifolds of co-dimension two. Precisely, they show that in the $\epsilon\to0$ limit, the energy measure of critical points with uniformly bounded energy converge sub-sequentially to an integer rectifiable stationary varifold $V$ of co-dimension two. Moreover they show that the currents dual to the curvature two-form converge to an integer rectifiable $(n-2)$-cycle $\Gamma$ with $|\Gamma|\leq\mu_V$.
 
 As a first step towards understanding the quantitative behavior of minimizers of the $\epsilon$-rescaled energy in higher dimensions and regularity properties of the vortex set via a blow-up analysis, it is necessary to have a complete understanding of the stability of the energy \cref{energy} in two dimensions for arbitrary pairs. In fact, these estimates are mainly motivated by their importance in \cite{DHP}. Roughly speaking, for an almost flat critical point of the re-scaled functional \cref{rescaled-energy-definition} in dimension $n\geq 3$, transversal $2$-dimensional slices nearly minimize the two dimensional energy, with an error quantified by the \textit{flatness}.
 
 For any sharp functional inequality it is also natural to ask \textit{"Can we estimate the distance to critical points by the discrepancy between the left and right hand side for some appropriate notion of distance?"}.

    For instance this question has been extensively studied for the isoperimetric inequality (see \cite{Maggi-2,selection-principle,Maggi-3,quantitative-isoperimetric-1,quantitative-isoperimetric,Maggi-5,Maggi-6,Maggi-1}) via methods of PDE and symmetrization. In \cite{selection-principle} Cicalese and Leonardi use a penalization technique and regularity theory for quasi-minimizers of the perimeter to find uniform $C^1$ approximations of sets, for which they use PDE techniques to prove stability. The methods in this article are partly inspired from this approximation technique (see \cref{selection-principle-section}).

    We first observe in \cref{section-2} that the \textit{discrepancy} is
    \begin{align*}
        E(u,\nabla) - 2\pi N = \int_{\R^2} r^2|\star d \log(r) + A-d\theta|^2 + |\star dA - \frac{1-r^2}{2}|^2\,,
    \end{align*}
    where $r = |u|$ and $A$ is the real connection one-form of $\nabla: d - iA$. This leads us to investigate the \textit{perturbed vortex equations}:
    \begin{align}\label{rescaled-energy-definition}
        \star d\log{|u|} + A -d\theta = \frac{f_1}{|u|} \text{  and  } \star dA -\frac{1-|u|^2}{2} = f_2\,,
    \end{align}
    for some $f_1,f_2$ in $L^2(\R^2)$. The main difficulty is the error term $\frac{f}{|u|}$ for which Muckenhoupt theory \cite{Coifman-Fefferman} and Caffarelli-Kohn-Nirenberg inequalities \cite{CKN} fall short. \L{}ojasiewicz inequalities are also a possible technique (as used in classical GMT applications e.g. by L. Simon in \cite{Simon-1}). However obtaining the inequality with the sharp power (such as the techniques of Topping in \cite{Topping}) is rather difficult. However we are able to improve the existence and known results to the following sharp stability:
    \begin{thm}\label{main-result}
        For any integer $N$ there exists a constant $C_{|N|}>0$ such that for any section and connection $(u,\nabla)\in W^{1,2}_{loc}(\R^2)$ on the trivial line bundle $L=\R^2\times\C\rightarrow \R^2$ with $\deg(u)=N$ and small enough discrepancy $E(u,\nabla)-2\pi|N|$ we have that:
            \[\min_{(u_0,\nabla_0) \in \mathcal{F}}\|u-u_0\|^2_{L^2(\R^2)} + \|F_{\nabla}-F_{\nabla_0}\|^2_{L^2(\R^2)} \leq C_{|N|} \left[E(u,\nabla)-2\pi |N|\right]\,,\]
        where (up to a conjugation) $\mathcal{F}$ is the family of all $N$-vortex minimizers of the Yang-Mills-Higgs energy.
    \end{thm}
    The proof relies on two main tools: Weighted elliptic inequalities of \cite{Halavati-inequality} (Recalled in \cref{inequalities-appendic-section}), in particular the weighted Hodge decomposition and a selection principle (inspired by \cite{selection-principle}) using a penalized functional (see \cref{selection-principle-section}). Roughly speaking, this method is analogous to running the gradient flow for unit time. However the proof of existence for minimizers of the penalized functional \cref{auxilary-energy} is straightforward as apposed to existence of the gradient flow (especially on an unbounded domain).
    
    As a by product of these methods, we also prove a weighted Sobolev stability for \textit{regular} enough pairs in the following theorem:
    \begin{thm}\label{main-result-epsilon-sobolev}
        For any $\Lambda>1$ and integer $N$, there exists constants $C_{\Lambda,|N|},\eta_{\Lambda,|N|}>0$ with the following property. Let $(u,\nabla)\in W^{1,2}_{loc}(\R^2)$ be an $N$-vortex section and connection such that
        \begin{enumerate}[label=(\roman*)]
            \item $\star d\left( (\frac{u}{|u|})^*(d\theta)\right) = 2\pi\sum_{k=1}^{|N|}\delta_{x_k}$ for a collection of points $\{x_k\}_{k=1}^{|N|} \subset \R^2$ counted with multiplicity.
            \item $E(u,\nabla)-2\pi|N| \leq \eta_{\Lambda,|N|}^2$,
            \item $\Lambda^{-1} |u_0|\leq |u|\leq \Lambda |u_0|$ for some $N$-vortex solution $(u_0,\nabla_0)$ with $\{x_k\}_{k=1}^{|N|}$ as the zero set (counted with multiplicity).
        \end{enumerate}
        Then for any $0<\epsilon<\frac1N$
        \begin{align*}
            \int_{\R^2}|u_0|^{2+2\epsilon}\left[ \left|d\log\left(\frac{|u|}{|u_0|}\right)\right|^2 +|A_0-A|^2\right] \leq \frac{C_{\Lambda,|N|}}{\epsilon^2} \left[E(u,\nabla)-2\pi|N|\right]\,,
        \end{align*}
        up to a choice of $\pm A$.
    \end{thm}
    \begin{rmk*}
        The assumption \textit{(i)} is satisfied if $\frac{u}{|u|} \in W^{1,1}_{loc}(\R^2)$. This is in fact a direct consequence of \cite[Theroem 1.2]{jerrard-soner-1}. However it is not clear if \textit{(i)} can be inferred from \textit{(ii)} and $(u,\nabla)\in W^{1,2}(\R^2)$.
    \end{rmk*}
    
    A central tool in the analysis of the abelian Higgs model in any dimension $n$ is the Yang-Mills-Higgs Jacobian $J(u,\nabla)$, which is a two-form defined as follows:
    \begin{align*}
        J(u,\nabla)(j,k):= \ang{2i\nabla_j u,\nabla_ku} + \omega(j,k)(1-|u|^2)\,, \text{ for all } 1\leq j<k \leq n\,,
    \end{align*}
    where $\omega$ is the real curvature two-form associated to the connection $\nabla$. It is the analogue of the Jacobian in the Ginzburg-Landau model (see \cite{jerrard-soner}). In \cref{jacobian-regular-stability-section} using \cref{main-result-epsilon-sobolev}, we prove the second order stability of the Jacobian for \textit{regular} pairs:
    \begin{thm}\label{jacobian-main-result}
        For any $\Lambda>1$ and integer $N$ there exists constants $C_{\Lambda,|N|},\eta_{\Lambda,|N|}>0$ such that for any $(u,\nabla)\in W^{1,2}_{loc}(\R^2)$ satisfying \textit{(i),(ii)} and \textit{(iii)} in \cref{main-result-epsilon-sobolev} the following estimate holds:
        \begin{align*}
             \int_{\R^2} \left|J(u,\nabla) - J(u_0,\nabla_0)\right| \leq C_{\Lambda,|N|}\sqrt{E(u,\nabla)-2\pi|N|}\,,
        \end{align*}
        up to a conjugation of $u$.
    \end{thm}
    
    To prove an improvement of flatness type result in \cite{DHP}, we rely on energy comparison with the pull-back of the two dimensional solution via a suitable harmonic competitor. The key difficulty there is to attach the boundary data which uses \cref{main-result}, \cref{main-result-epsilon-sobolev} and \cref{jacobian-main-result} as the main ingredients (See \cite[Proposition 10.2]{DHP}).

    For a non-optimal \L{}ojasiewicz-type inequality for Yang-Mills-Higgs energy in complex geometry see \cite[Section 1.3.3]{Feehan-Maridakis-Application-ymh}. Related to the estimates in this paper, in \cite[Section 3]{stuart}, D.Stuart defines a corrected Hessian for the Yang-Mills-Higgs energy (to mod out the gauge freedom) and derives coercive estimates. Further studies on the Yang-Mills-Higgs energy on K\"ahler manifolds has been done in \cite{Heat-flow-Hong,Li-Xi,bradlow}. For stability type results on the Yang-Mills functional we refer the reader to \cite{Taubes-3,Waldron}. It is also worthy to mention the articles \cite{Pigati-1,Pigati-2} for results connecting the variational theory of the re-scaled Yang-Mills-Higgs functional and minimal surfaces and \cite{serfaty} for the magnetic Ginzburg-Landau theory.

    At the present we do not know how to obtain the estimates of \cref{main-result-epsilon-sobolev} for the case $\epsilon=0$. The reason is that the embeddings of \cref{hodge-prop} are no longer compact for $\epsilon=0$. The stability in this problem is also deeply related to Poincar\'e inequalities on the unbounded one-sided cylinder $\R^{+}\times S^1$ via the log-polar coordinate, which fail to be true. However by increasing the power in the weight by a factor of $\epsilon$, analogously, assigning an exponentially decaying weight on the height $e^{-\epsilon t}:(t,\theta)\in \R^+\times S^1 \to \R^+$, we are roughly compactifying the domain with a total measure of order $\epsilon^{-2}$. This also serves as an intuition for the factor in the right hand side of \cref{main-result-epsilon-sobolev}. However the power $2$ is sharp. We also state the following problem:
    \begin{problem*}
        Do the estimates of \cref{main-result-epsilon-sobolev} fail for the case $\epsilon=0$? In particular, does there exists a sequence of $N$-vortex pairs $\{(u_k,\nabla_k)\}_{k=1}^{\infty}$ such that the sharp stability fails in the following sense
        \begin{align*}
            \lim_{k\to\infty} E(u_k,\nabla_k) = 2\pi|N| \text{ and }
            \lim_{k\to\infty} \frac{E(u,\nabla) - 2\pi|N|}{\min_{(u_0,\nabla_0)\in\mathcal{F}} \|du - du_0\|_{L^2(\R^2)}^2}=0\,.
        \end{align*}
    \end{problem*}
    
    \subsection{Results on compact surfaces} The Bogomolny trick also works on nontrivial line bundles $L\to M$ over closed surfaces. In this case the energy is lower bounded by the degree of the line bundle:
    \begin{align}
        E(u,\nabla) \geq 2\pi |\deg(L)|\,.
    \end{align}
    In this case the \textit{vortex equations} take the same form:
    \begin{align}
        \nabla_{\de_1} u + i\nabla_{\de_2} u = 0 \text{ and }
        \star F_\nabla = \frac{1-|u|^2}{2}\,.
    \end{align}
    Integrating the second equation over $M$, we see that:
    \begin{align*}
        |\deg(L)| \leq \frac{1}{4\pi} \vol(M)\,.
    \end{align*}
    In \cite{GarcaPrada1994ADE} Garc\'ia-Prada obtained the analogues existence and uniqueness (for a slightly generalized functional) provided that the above constraint is satisfied. In this article we also prove the analogues stability, in the following theorem:
    \begin{thm}\label{main-result-manifold}
        Let $(M,g)$ be a smooth compact Riemann surface and $L\to M$ a Hermitian line bundle over $M$ with $|\deg(L)| \leq \frac{1}{4\pi} \vol(M)$. Then there exists a constant $C_{M}>0$ with the following property: Let $(u,\nabla)$ be a section and connection on $L$ such that $E(u,\nabla) - 2\pi|\deg(L)|$ is small enough, then:
        \begin{align*}
            \min_{(u_0,\nabla_0)\in \mathcal{F}} \|u-u_0\|_{L^2(M)}^2 + \|F_\nabla - F_{\nabla_0}\|_{L^2(M)}^2 \leq C_{M}\left[E(u,\nabla) - 2\pi|\deg(L)|\right]\,,
        \end{align*}
        where (up to a conjugation) $\mathcal{F}$ is the family all minimizers of the Yang-Mills-Higgs energy on $L\to M$.
    \end{thm}

    Similarly for compact Riemann surfaces, we have the following result for \textit{regular} pairs:
    \begin{thm}\label{main-result-epsilon-sobolev-manifold}
        Let $(M,g)$ be a smooth compact Riemann surface and $L\to M$ a Hermitian line bundle over $M$ with $|\deg(L)| \leq \frac{1}{4\pi} \vol(M)$. For any $\Lambda>1$, there exists $C_{\Lambda,M},\eta_{\Lambda,M}>0$ with the following property. Let $(u,\nabla)\in W^{1,2}_{loc}(M)$ be a pair such that
        \begin{enumerate}[label=(\roman*)]
            \item $\star d\left( (\frac{u}{|u|})^*(d\theta)\right) = 2\pi\sum_{k=1}^{|\deg(L)|}\delta_{x_k}$ for a collection of points $\{x_k\}_{k=1}^{|\deg(L)|} \subset M$ counted with multiplicity.
            \item $E(u,\nabla)-2\pi|\deg(L)| \leq \eta_{\Lambda,M}^2$,
            \item $\Lambda^{-1} |u_0|\leq |u|\leq \Lambda |u_0|$ for some solution $(u_0,\nabla_0)$ with $\{x_k\}_{k=1}^{|\deg(L)|}$ as the zero set (counted with multiplicity).
        \end{enumerate}
        Then for any $0<\epsilon<\frac1N$:
        \begin{align*}
            \int_{M}|u_0|^{2+2\epsilon}\left[ \left|d\log\left(\frac{|u|}{|u_0|}\right)\right|^2 + |A_0-A|^2\right] \leq \frac{C_{\Lambda,M}}{\epsilon^2} \left[E(u,\nabla)-2\pi|\deg(L)|\right]\,.
        \end{align*}
         up to a conjugation of $u$. Moreover the following inequality holds for the Jacobian:
        \begin{align*}
            \int_{M} \left|J(u,\nabla) - J(u_0,\nabla_0)\right| \leq C_{\Lambda,M}\sqrt{E(u,\nabla)-2\pi|\deg(L)|}\,.
        \end{align*}
    \end{thm}
    
    \subsection{Outline of the proof} Here we give an overview of the plan of the paper:
        
        \fbox{\textbf{Step 1.}} In \cref{section-2} we prove that the degree is well defined for pairs with finite energy. Then we re-derive the first order vortex equations \cref{vortex-eq} and subsequently the PDE \cref{vortex-PDE} and we define the \textit{discrepancy} \cref{vortex-discrepancy}. Then in \cref{vortex-set-estimates} we show that solutions are locally comparable to $\Pi_{k=1}^{M}|x-a_k|$.
        
        \fbox{\textbf{Step 2.}}
        Here we explain the case of just one vortex for the sake of clarity. The ideas carry over to the case of multi vortex situations, since the elliptic and Poincar\'e type inequalities of \cite{Halavati-inequality} (recalled in \cref{inequalities-appendic-section}) have uniform and explicit constants. Now assume that $u = e^{h}u_0$ for a compactly supported real valued $h\in C^{\infty}_c(B_1(0))$ and a one-vortex solution $u_0$ centered at the origin. In this case after a suitable gauge transformation we can linearize the discrepancy $E(u,\nabla) - 2\pi = \eta^2$ as follows
        \begin{align*}
            \eta^2 \sim \int_{B_1}|x|^2|\star dh + B|^2 + |\star dB + V(x)h|^2\,,
        \end{align*}
        where $h = \log\left(\frac{|u|}{|u_0|}\right)$ and $B = A-A_0$ and $C^{-1}|u_0|^2 \leq V(x) \leq C|u_0|^2$.
        
        Now if the right hand side is zero, the first term tells us that $\star dh = -B$. Then we substitute this in the second term to see that $-\Delta h + V(x)h = 0$; testing this PDE with $h$ and integrating by parts, we conclude that $h = 0$.

        Now we aim to make this quantitative. First it is instructive to see the compactness argument if the first term was not weighted: Arguing by contradiction and scaling, we can assume that $\|h\|_{L^2}=1$ while $\eta\rightarrow0$. Then we see that $\star dh$ is uniformly close in $L^2$ topology to the co-exact part of $-B$, for which we have uniform $W^{1,2}$ bounds using the second term. Then we conclude that $h$ (up to extracting a sub-sequence) converges to zero, strongly in $L^2$. The idea is to adapt this proof to the weighted case.
        
        For this purpose we introduce a $|x|^2$-\textit{weighted Hodge decomposition} of $B$ as the minimizer of the following weighted functional
        \begin{align*}
            \inf_{v \in W^{1,2}_0(|x|^2,B_1)}\int_{B_1} |x|^2|B-\star dv|^2\,,
        \end{align*}
        where $W^{1,2}(|x|^2,B_1)$ is the $|x|^2$-weighted Sobolev space (for details, see \cref{inequalities-appendic-section}). With the direct method in the calculus of variation and \cref{Generalized-CKN-thm} we guarantee the existence of a minimizer. Then the Euler Lagrange equations of minimizers tells us that:
        \begin{align*}
            |x|B = \star |x|dv + |x|^{-1}df \text{ and } B = \star dp + dq\,.
        \end{align*}
        Then we rewrite the linearized discrepancy using these identities. To be more precise we use the weighted Hodge decomposition for the first term and the standard Hodge decomposition for the second term:
        \begin{align}\label{decomposed-discrepancy-example}
            \eta^2 \sim \int_{B_1} |x|^2|d(h+v)|^2 +|x|^{-2}|df|^2 + |\Delta p + |u_0|^2h|^2 \,.
        \end{align}
        Then we apply the weighted elliptic estimates of \cite{Halavati-inequality} recalled in \cref{hodge-prop} and we estimate the distance of the weighted and the standard Hodge decomposition as follows:
        \begin{align*}
            \int_{B_1} |x|^{2+2\epsilon}|d(v-p)|^2 \leq \frac{C}{\epsilon^2} \int_{B_1}|x|^{-2}|df|^2 \leq \frac{C}{\epsilon^2} \eta^2\,,
        \end{align*}
        for any $\epsilon > 0$. With this and \cref{decomposed-discrepancy-example} we get the uniform $W^{1,2}(B_1)$ estimates needed for $|x|^{1+\epsilon}h$ near the vortex set to gain compactness.

        To generalize this heuristic to many vortices, we note that all weights of the form $\Pi_{k=1}^M|x-a_k|^{\alpha_k}$ with $\alpha_k>0$ satisfy the condition \cref{weight-weak-condition} to apply \cref{Generalized-CKN-thm}.
        
        Then with a concentration compactness type argument we \textit{glue} the local weighted estimates near the vortex set and uniform elliptic estimates far from the vortex set to conclude the stability for \textit{regular} sections.

        \fbox{\textbf{Step 3.}} In \cref{selection-principle-section} we generalize the stability to require only that the pair $(u,\nabla)\in W^{1,2}_{loc}(\R^2)$ nearly minimizes the energy. We use a \textit{selection principle} and construct a new pair $(u_1,\nabla_1)$ by a finite iterative process of replacing $(u,\nabla)$ with the minimizer of the auxiliary functional
        \begin{align*}
            E(u_1,A_1) + \| u_1 - u \|_{L^2(\R^2)}^2 + \| A_1 - A \|_{L^2(\R^2)}^2\,.
        \end{align*}
        By \cref{auxilary-energy-regularity} we gain some regularity at each step so after finitely many steps we have a pair $(\tilde u,\tilde\nabla)$ with uniform $C^{N,\alpha}$ bounds in the local Coulomb gauge. Arguing by contradiction and Arzela-Ascoli we conclude that with small enough discrepancy $u_1$ is close enough to $u_0$ in $C^N$ topology. Since $u_0$ is analytic, we see that $u_1$ is a $C^N$ perturbation of a complex polynomial with degree $\leq N$. We then apply \cref{complex-polynomial-perturbration} to see that $C^N$ perturbations of complex polynomials with degree $M\leq N$ are uniformly comparable to another complex polynomial. This reduces the problem to \cref{regular-stability} in \cref{regular-stability-section}.

        \fbox{\textbf{Step 4.}} We show that the methods above can be adapted, with little to no modification, to prove stability for nontrivial line bundles over arbitrary smooth compact Riemann surfaces. Here instead of polynomials, we use weights of the form $\Pi_{k=1}^n e^{-\alpha_k G_{x_k}(x)}$ with $\alpha_k >0$, where $G_x(y)$ is the Green's function for a domain $\Omega\subset M$ in a Compact Riemann surface $M$. Note that in two dimensions, the Green's function for the Laplacian is proportional to $-\log{d(x,y)}$, where $d(x,y)$ is the geodesic distance between $x,y$ on $M$; so essentially we work with weights proportional to $\Pi_{k=1}^n d(x,x_k)^{\alpha_k}$. Notice that estimates of \cref{inequalities-appendic-section} work with universal constants on any surface with boundary.
    
    \section{The vortex equations}\label{section-2}
                We work on Hermitian line bundles over smooth manifolds; on the trivial bundle $L=\C\times \R^2$,
                we can always write a metric connection $\nabla$ as
                $$\nabla=d-iA,$$
                for a real-valued one-form,
                meaning that $\nabla_\xi s=ds(\xi)-i\alpha(\xi)s$.
                
                In general, for two vector fields $\xi$ and $\eta$, typically $\nabla_\xi$ and $\nabla_\eta$ do not commute, meaning that the connection has nontrivial \emph{curvature}. Formally, the curvature $F_\nabla$ is given by
                \begin{align}\label{curvature-definition}
                    F_\nabla(\xi,\eta) (s) = [\nabla_\xi, \nabla_\eta] s - \nabla_{[\xi,\eta]} s.
                \end{align}
                A simple computation shows that $F_\nabla$ is a two-form with values in imaginary numbers; we will sometimes use the real-valued two-form $\omega$ given by
                \begin{align}\label{omega}
                    F_\nabla(\xi,\eta) (s)=:-i\omega(\xi,\eta)s.
                \end{align}
                On the trivial bundle, if $\nabla=d-iA$ then we simply have
                $$ \omega=dA. $$ 
                Notice that the Yang-Mills-Higgs energy functional \cref{energy} enjoys the gauge invariance
        \begin{align*}
            (u,\nabla) \rightarrow (u e^{i\xi} , \nabla - id\xi)\,,
        \end{align*}
        for any compactly supported function $\xi \in C^\infty_c(\R^2)$. Moreover, after fixing the connection one-form $\nabla: d-iA$ we can rewrite the energy \cref{energy} as follows:
        \begin{equation*}
            E(u,A)=\int_{\R^2} |du-iu\otimes A|^2 + |dA|^2 + \frac{(1-|u|^2)^2}{4}\,.
        \end{equation*}
        The one-form $A$ is sometimes called the \textit{magnetic vector potential}. Observe that if $u=re^{i\theta}$ the energy can also be written in the following form:
        \begin{equation}\label{energy-1}
            E(re^{i\theta},A) = \int_{\R^2} |dr|^2 + r^2|A-d\theta|^2 + |dA|^2 + \frac{(1-r^2)^2}{4}\,.
        \end{equation}
        Note that $\theta$ cannot be defined globally, however $d\theta$ is defined by pulling back the tangent vector to $S^1$ by the map $\frac{u}{|u|}$.
        
        In \cite{Taubes} C.H.Taubes proves existence and uniqueness for minimizers of \cref{energy} using the \textbf{vortex equations} which are as follows (up to a conjugation or a change of orientation):
        \begin{equation}\label{vortex-eq}
            \star dr = -r(A-d\theta) \;\;\; \star dA = \frac{1-r^2}{2}\,.
        \end{equation}
        \subsection{The degree and the discrepancy for finite energy pairs}
        In the following lemma using the trick of Bogomolny \cite{Bogomolny} we first prove that pairs $(u,\nabla)\in W^{1,2}_{loc}(\R^2)$ with finite energy have globally well-defined degree. Moreover we also derive the \textit{discrepancy} and the \textit{vortex equations}.
        \begin{lemma}\label{IBP}
            Let $(u,\nabla)\in W^{1,2}_{loc}(\R^2)$ be a pair of section and connection on the trivial line bundle $\C\times\R^2$ with finite energy \cref{energy} $E(u,\nabla)=\Lambda <\infty$ and $u=re^{i\theta}$ (for some $\theta:\R^2\rightarrow S^1)$ and $\nabla: d-iA$. Then:
            \begin{enumerate}[label=(\roman*)]
                \item The degree of $u$, namely $\deg(u)=N$ is globally well defined.
                \item The integral in \cref{energy} or equivalently \cref{energy-1} (possibly after a conjugation of $u$) can be rewritten as:
                \begin{align}\label{vortex-discrepancy}
                    E(re^{i\theta},A) &= 2\pi |N| + \int_{\R^2} |\star dr + r(A-d\theta)|^2 + |\star dA-\frac{1-r^2}{2}|^2\,.
                \end{align}
            \end{enumerate}
            
            \begin{proof}
                We use the notation $u=re^{i\theta}$ and $\nabla=d-iA$ as defined above. Note that $(u,\nabla) \in W^{1,2}_{loc}(\R^2)\subset \mathrm{VMO}_{loc}(\R^2)$, which is the space of functions with locally vanishing mean oscillation. We know from \cite[II.2 Property 2]{Brezis-Nirenberg-degree} that the degree is locally well defined. We need to show that it is also globally well-defined. First name the sub-level set $Z_{1/2}=\{r<1/2 \}$ with the disjoint open and connected components $Z_{1/2} = \bigcup_{j=1}^{\infty} \Omega^{1/2}_{j}$. By the Coarea formula we have that:
                \begin{align*}
                    \int_{0}^{1/2}\sum_{j=1}^{\infty}\hau^1\big( \{r=t\}\cap \Omega^{1/2}_j \big) dt &= \int_{Z_{1/2}} |dr| \leq C\int_{Z_{1/2}} |dr|^2 + \frac{(1-r^2)^2}{4} \leq C\Lambda
                \end{align*}
                Then by the mean value theorem we can find some threshold $\frac14 < \beta < \frac12$ such that
                $\sum_{j=1}^{\infty} \hau^1(\de\Omega^{\beta}_j) < C\Lambda\,.$
                Here $Z_{\beta} = \bigcup_{j=1}^{\infty} \Omega^{\beta}_j$ are the disjoint connected components of $Z_\beta=\{|u|<\beta\}$. Now since perimeter bounds diameter we can see that:
                \begin{align*}
                    \sum_{j=1}^{\infty} \textrm{diam}(\Omega^{\beta}_j) < C\Lambda\,.
                \end{align*}
                Since each $\Omega^{\beta}_j$ has finite diameter, we can see that the degree $\deg(u,\de\Omega_j^{\beta})$ is well defined on each domain. Now we proceed with a Cauchy-Schwartz and Stokes theorem:
                \begin{align*}
                    \infty >\Lambda &\geq \sum_{j=1}^{\infty}\int_{\Omega^{\beta}_j} |dr|^2 + r^2|A-d\theta|^2 + |dA|^2 + \frac{(1-r^2)^2}{4}  \\&\geq  C\sum_{j=1}^{\infty}\left|\int_{\Omega^{\beta}_{j}} dA\frac{(\beta^2-r^2)}{2} - r\star dr\wedge(A-d\theta)\right| \\&=C\sum_{j=1}^{\infty}\left|\int_{\Omega^{\beta}_{j}}  dA\frac{(\beta^2-r^2)}{2} +\star d(\beta^2-r^2)\wedge(A-d\theta)\right| \\&=C\sum_{j=1}^{\infty}\left|\int_{\Omega^{\beta}_j}d(d\theta)(\beta^2 -r^2)\right|=C\beta^2\sum_{j=1}^{\infty}|\deg(u,\de\Omega_j^{\beta} )|
                \end{align*}
                The last line follows from $\int_{\Omega_j^\beta} d(d\theta) = \int_{\de\Omega_j^\beta} \de_\tau\theta = \deg(u,\Omega_j^\beta)$. This tells us that only finitely many of the domains $\Omega^\beta_j$ have non-zero degree. Hence we can see that the degree is also globally well-defined. We also name $$\deg(u)=N\,.$$

                To prove the second point, we again replace $|dr|^2$ with $|\star dr|^2$ and complete the squares in \cref{energy-1}:
                \begin{align*}
                    E(re^{i\theta},A)&= \\ = \int_{\R^2}& |\star dr|^2 + r^2|A-d\theta|^2 + |dA|^2 + \frac{(1-r^2)^2}{4}\\
                    =\int_{\R^2}& |\star d r + r(A-d\theta)|^2  + |\star dA - \frac{1-r^2}{2}|^2  - 2r d r \wedge (A-d\theta) + dA(1-r^2)\,.
                \end{align*}
                Using $d(1-r^2) = - 2rd r$ and Stokes theorem and noticing that $$\int_{\R^2}- 2r  d r \wedge (A-d\theta) +  dA(1-r^2) \leq C E(u,\nabla) < \infty\,,$$ we see that:
            \begin{align*}
                E(re^{i\theta},A)= \int_{\R^2}|\star dr + r(A-d\theta)|^2  + |\star dA - \frac{1-r^2}{2}|^2 + d(d\theta)(1-r^2)\,.
            \end{align*}
            For any compact set $K\in \R^2\backslash\{r=0\}$ we see that $d\theta \in W^{1,2}(K)$ hence $d(d\theta)$ is supported on the zero set of $r$ we can see that:
            \begin{align*}
                \int_{\R^2} d(d\theta)(1-r^2) = \sum_{j=1}^{\infty} \int_{\Omega_j^\beta}d(d\theta)(1-r^2) = \sum_{j=1}^{\infty} \int_{\Omega_j^\beta}d(d\theta) = 2\pi \deg(u)\,.
            \end{align*}
            Then after a possible conjugation of $u$ we see that
            \begin{align*}
                E(re^{i\theta},A) &= 2\pi |\deg(u)| + \int_{\R^2}|\star d r + r(A-d\theta)|^2  + |\star dA - \frac{1-r^2}{2}|^2\,.
            \end{align*}
            We get equality if and only if:
            \begin{align*}
                \star d r = -r(A-d\theta) \;\;\;\; \star dA = \frac{1-r^2}{2}\,.
            \end{align*}
            These are called the first order \textit{vortex equations}.
            \end{proof}
        \end{lemma}
        \begin{rmk}
            Assuming sufficient decay for $|\nabla u|(x)$ as $|x|\to\infty$ (see the conditions in \cite[Chapter II, Theroem 3.1]{Taubes-2}) we can also see that:
            \begin{align*}
                \int_{\R^2}dA = 2\pi N\,.
            \end{align*}
            For a detailed discussion see \cite[Chapter II.3]{Taubes-2}. But here we do not need this result. Notice that even if $dA$ is sufficiently close to an integral two-form in $L^2(\R^2)$, since $\R^2$ is unbounded, one can draw no conclusions on the integrality of $dA$.
        \end{rmk}
    \subsection{Estimates on solutions of the vortex equations}
    In the sequel without loss of generality, we assume that $N\geq0$ after a possible conjugation of $u$. Here we first prove that solutions of the vortex equations are comparable to modulus of polynomials on their sublevel sets. We collect these information in the following lemma:
    \begin{proposition}\label{vortex-set-estimates}
            There exists a constant $C_N>1$ depending only on $N\geq0$ with the following property: Let $(u_0,\nabla_0)$ be a solution to the vortex equations \cref{vortex-eq} as in \cite{Taubes} with the prescribed zero set $x_1,\dots,x_N \in \R^2$ counted with multiplicity. Then there exists $M\leq N$ balls $\{B_{\rho_k}(z_k)\}_{k=1}^{M}$ and some $\frac14 \leq \beta \leq \frac12$ such that:
            \begin{enumerate}[label=(\roman*)]
                    \item $Z_\beta = \{|u_0| \leq \beta\} \subset \bigcup_{k=1}^{M} B_{\rho_k}(z_k)$\,,
                    \item $B_{2\rho_i}(z_i) \cap B_{2\rho_j}(z_j) = \emptyset $ for all $1\leq i < j \leq M$\,,
                    \item $1 \leq \rho_k \leq C_N$ for all $1\leq k \leq M$\,,
                    \item $C_N^{-1} \omega_k \leq |u_0| \leq C_N \omega_k$ in $B_{2\rho_k}(z_k)$, s.t. $\omega_k(x) = \Pi_{x_i \in B_{\rho_k}(z_k)} |x-x_i|$ for all $1 \leq k \leq M$.
            \end{enumerate}
            \begin{proof}
                First we estimate the measure of $\{|u|\leq\frac12\}$. Then we use the co-area formula and a mean value theorem to find a sub-level set $\{|u|\leq \beta\}$ set with bounded total perimeter. Then we cover them with non intersecting balls and use \cref{vortex-PDE} and the maximum principle.
                
                Name $r_0=|u_0|$; by the \textit{vortex equations}, we know that $r_0$ is a solution of the following
                 \begin{align}\label{vortex-PDE}
                    - \Delta \log(r_0) + \frac12(r_0^2 -1 ) = -\sum_{i=1}^N 2\pi \delta_{x_i}\,,
                 \end{align}
                 where $\{x_1,\dots,x_N\} \subset \R^2$ is the zero set of $r_0$ (counted with multiplicity) and $\delta_{x}$ is a point-mass on $x\in \R^2$. Define the sub-level set and its disjoint connected components $Z_\beta = \{r_0 \leq \beta\} = \bigcup_{j=1}^\infty Z^j_\beta$. Then we multiply \cref{vortex-PDE} by $r_0^2 - \beta^2$ and integrate by parts on each $Z_\beta^j$:
                 \begin{align*}
                    0 \leq \int_{Z^j_\beta} 2|dr_0|^2 + \frac12(1-r_0^2)(\beta^2 - r_0^2) = 2\pi\beta^2 K_j\,.
                 \end{align*}
                 Here $K_j$ is the number of zeros in $Z_\beta^j$; if $K_j=0$ we see that $r_0 = \beta$ on $Z^j_\beta$ and since solutions of the vortex equations are analytic (by \cite[Proposition 6.1]{Taubes}), unique continuation tells us that $r_0=\beta$ globally, which is a contradiction if $\beta<1$ or $N\not=0$; so we conclude that $K_j \geq 1$ for finitely many $j$s. This means that there are at most $N$ connected components of $Z_\beta$ for all $\beta < 1$.
                 
                 Now we sum the estimates on $Z_{\frac34}$:
                 \begin{align*}
                    |Z_{\frac12}| + \int_{Z_{\frac12}} |dr_0|^2 \leq  C \int_{Z_{\frac34}} 2|dr_0|^2 + \frac12(1-r_0^2)(\frac{9}{16} - r_0^2) \leq  C_N\,.
                 \end{align*}
                 By the co-area formula and mean-value theorem we get that there exists some $\frac14 \leq \beta \leq \frac12 $ such that:
                 \begin{align*}
                    \mathcal{H}^1(\de Z_\beta ) \leq C \int_{\frac14}^{\frac12} \mathcal{H}^1(\de Z_s) ds \leq C\int_{Z_{\frac12}} |dr_0| \leq C|Z_{\frac12}|^{\frac12} \big(\int_{Z_{\frac12}}|dr_0|^2)^{\frac12} \leq C_N\,.
                 \end{align*}
                 Since the set $Z_\beta$ has at most $N$ connected components $\{Z^j_\beta\}_{j=1}^M$ for some $M \leq N$ we get that $\max_{1\leq j \leq M}\diam(Z^j_\beta) \leq C_N$. Then we find $M$ balls $B_{r}(z_1),\dots,B_{r}(z_M)$ with $1\leq r \leq C_N$ whose union covers $Z_\beta \subset \cup_{i=1}^{M} B_{r}(z_i)$ and $z_i \in Z_\beta$. To find the balls covering $Z_\beta$ satisfying the assumptions, at each step we replace any two balls $B_{r_i}(z_i),B_{r_j}(z_j)$ that, if dilated, intersect $B_{2r_i}(z_i) \cap B_{2r_j}(z_j) \not= \emptyset$ with the ball $B_{3(r_i+r_j)}(\frac{z_i+z_j}{2})$. Since at each step the number of balls decreases, this procedure stops at maximum $N-1$ steps. Notice that in each step $\max(r_i)$ increases at most by a factor of $6$ so we are left with $M\leq N$ balls $B_{\rho_1}(z_i),\dots,B_{\rho_M}(z_M)$ such that $1 \leq \rho_i \leq C_N$.

                 Now for each $1 \leq j \leq M$ define the weight $\omega_j = \Pi_{i=1}^{K_j}|x-x_i|$ where $x_1,\dots,x_{K_j}$ are the zeros of $r_0$ in the ball $B^j_{\rho_j}(z_j)$. Then we estimate $\|\log(\omega_j) - \log(r_0)\|_{L^\infty(B_{2\rho}^j(z_j))}$ using \cref{vortex-PDE}:
                 \begin{align*}
                    -\Delta h &= \frac12(1-r_0^2) \text{ in } B^j_{2\rho_j}(z_j)\,,\\
                    h &= \log(r_0) - \log(\omega_j)\,.
                 \end{align*}
                 Notice that by \cref{vortex-PDE} we have $\Delta r_0^2 = 4|dr_0|^2 - r_0^2(1-r_0^2)$ and with an application of the maximum principle and $\int_{\R^2}|dr_0|^2 < \infty$ we can see that $r_0 \leq 1$. We then notice that:
                 \begin{align*}
                     -\Delta h \geq 0 \text{ and } -\Delta (h - \rho_j^2 + \frac14|x-z_j|^2) \leq 0 \text{ in } B^j_{2\rho_j}(z_j)\,.
                 \end{align*}
                By the weak maximum principle we get the estimates below:
                 \begin{align*}
                     \|h\|_{L^\infty(B^j_{2\rho_j}(z_j))} \leq \rho_j^2 + \sup_{\de B^j_{2\rho_j}(z_j)} |h| \leq C_N^2 + \sup_{\de B^j_{2\rho_j}(z_j)} |\log(r_0)| + |\log(\omega_j)|\,.
                 \end{align*}
                From the bound on $\beta \geq \frac14$ we know that $\de B^j_{2\rho_j}(z_j) \subset Z_{\beta}^c \subset  \{r_0 > \frac14\}$ and this yields an upper bound on $|\log(r_0)|$ on the boundary of the ball. Since $\{x_1,\dots,x_{K_j}\} \in B^j_{\rho_j}(z_j)$ and $\rho_j \geq 1$, we get a lower bound for $\omega_j$ which combined with $\rho_j \leq C_N$ gives us an upper bound on $|\log(\omega_j)|$. Finally we estimate:
                 \begin{align*}
                     \|h\|_{L^\infty(B^j_{2\rho_j}(z_j))} \leq C_N \Rightarrow e^{-C_N} \omega_j \leq r_0 \leq e^{C_N} \omega_j \text{ in } B^j_{2\rho_j}(z_j)\,,
                 \end{align*}
                 and conclude.
            \end{proof}
        \end{proposition}
    \section{Stability of regular pairs}\label{regular-stability-section}
    In this section we show stability under some regularity conditions:
        \begin{thm}\label{regular-stability}
            For any $\Lambda>1$ and $N \in \mathbb{N}$ there exists $\eta_{\Lambda,N},C_{\Lambda,N} > 0$ with the following property. Let $(u,\nabla)$ be a section and connection on the trivial line bundle $\C\times\R^2 \rightarrow \R^2$ that satisfy:
            \begin{enumerate}[label=(\roman*)]
                \item $\star d\left( (\frac{u}{|u|})^*(d\theta)\right) = 2\pi\sum_{k=1}^{N}\delta_{x_k}$ for a collection of points $\{x_k\}_{k=1}^{|N|} \subset \R^2$ counted with multiplicity.
                    \item $E(u,\nabla) - 2\pi N \leq \eta_{\Lambda,N}^2$,
                    \item $\Lambda^{-1} |u_0| \leq |u| \leq \Lambda |u_0|$ for some $N$-vortex solution $(u_0,\nabla_0)$ with $\{x_k\}_{k=1}^{|N|}$ as the zero set (counted with multiplicity).
            \end{enumerate}
            Then:
            \begin{align*}
                \||u|-|u_0|\|^2_{L^2(\R^2)} + \|F_\nabla - F_{\nabla_0} \|^2_{L^2(\R^2)} \leq C_{\Lambda,N} \left[E(u,\nabla) - 2\pi N\right]\,.
            \end{align*}
        \end{thm}
        
        We divide the proof of this theorem to two parts. In the first part we deal with estimates near the vortex set and in the second part we combine these estimates with uniform elliptic estimates far from the vortex set.
        
        \subsection{Compactness estimates near the vortex set} This section contains the main ingredients of the paper, which are weighted inequalities near the vortex set, in the following proposition:
        \begin{proposition}\label{compactness-inequality-1}
        For any constant $\Lambda > 0 $, radius $R > 0$ and integer $N\in\mathbb{N}$ there exists a constant $C_{\Lambda,R,N}>0$ with the following property. For any function $h \in C^\infty_c(B_R)$, one-form $B\in C^\infty_c(\bigwedge^1 B_R)$ and weight $\omega$ such that
        \begin{align*}
            \omega(x) = \Pi_{i=1}^{M} |x-x_i| \text{ for } x_1,\dots,x_M \in B_R \text{ counted with multiplicity for } 1\leq M \leq N\,,
        \end{align*}
        we have the following inequality:
            \begin{align*}
                 \int_{B_R} \omega^2 |h|^2 \leq C_{\Lambda,R,N} \int_{B_R} \omega^2|\star dh + B|^2 + |\star dB + V(x)h|^2\,,
            \end{align*}
        provided that $0 \leq V(x) \leq \Lambda \omega(x)^{1+\frac1N}$.
        \begin{proof}
            Here we implicitly use the fact that all positive powers of $\omega$ as above are admissible weights for \cref{weight-weak-condition}. Note that for any weight $\omega(x) = \Pi_{i=1}^{M} |x-x_i|^{\alpha_i}$ with $\alpha_i>0$ the condition \cref{weight-weak-condition} is satisfied. For any $\phi\in C^{\infty}_c(\R^2)$
            \begin{align*}
                \int_{\R^2} \omega^2 \Delta \log(\omega)\phi &= \int_{\R^2}\omega(x)\sum_{i=1}^{M} \alpha_i\Delta\log|x-x_i|\phi \\
                &= 2\pi \int_{\R^2} \omega(x)\sum_{i=1}^M \alpha_i \delta_{x_i}\phi = 2\pi\sum_{i=1}^{M} \alpha_i\omega(x_i)\phi(x_i) = 0\,.
            \end{align*}
            The last line follows from $\omega(x_i)=0$ if $\alpha_i >0$.
            
            Now we divide the rest of the proof into 5 steps:
            
            \fbox{\textit{Step 1}.} We argue by contradiction. Assume there is a sequence $\{h_k,B_k,\omega_k,V_k\}_{k=1}^\infty$ satisfying the assumptions such that:
            \begin{equation*}
                \int_{B_R} \omega_k^2 |h_k|^2 = 1\,,
            \end{equation*}
            \begin{equation*}
                \delta_k^2 = \int_{B_R} \omega^2_k|\star dh_k + B_k|^2 + |\star dB_k + V_k(x)h_k|^2 \rightarrow 0 \text{ as } k\rightarrow \infty\,.
            \end{equation*}
            Now by \cref{hodge-prop} we write the standard Hodge decomposition and the \textit{weighted Hodge decomposition} for the one-form $B_k$:
            \begin{align}\label{hodges-1}
                B_k = \star dp_k + dq_k \text{ and } \omega_k B_k = \star \omega_k dv_k + \omega_k^{-1} df_k\,,
            \end{align}
            where $\omega_k^{-1}df_k \in L^2(\R^2)$, $q_k,\omega_kv_k \in W^{1,2}_0(B_R)$ and $p_k\in W^{2,2}_0(B_R)$; since $\Delta p_k = dB_k\in L^2(B_R)$.

            Now we rewrite $\delta_k^2$ using \cref{hodges-1} and Stokes theorem to see that:
            \begin{align}\label{temp-1}
                \delta_k^2 = \int_{B_R} \omega_k^2|d(h_k+v_k)|^2 + \omega^{-2}_k|df_k|^2 + |\Delta p_k + V_k(x) h_k|^2\,.
            \end{align}
            Notice that the term $dq_k$ has disappeared from the expression and this is consistent with the gauge invariance.
            
            By \textit{(ii)} in \cref{hodge-prop} for any $\epsilon>0$ we estimate:
            \begin{align}\label{halavati-1}
                \int_{B_R} \omega_k^{2+2\epsilon} |d(v_k-p_k)|^2 \leq C_{R,N} \epsilon^{-2} \int_{B_R} \omega_k^{-2} |df_k|^2 \leq C_{R,N} \frac{\delta_k^2}{\epsilon^2}\,.
            \end{align}
            Our goal is to use \cref{temp-1} to find uniform $W^{1,2}$ upper bounds on $\omega^{1+\epsilon}_kv_k$ (for small enough $\epsilon>0$) and a uniform non-zero lower bound on its $L^2$ norm to arrive at a contradiction.
            
            \fbox{\textit{Step 2}.} In this step we find a lower bound for the $L^{2}$ norm of $\omega^{1+\frac{1}{2N}}h_k$. First we show that there exists a positive constant $C_{R,N}>0$ such that for any $\epsilon < \frac1N$ we have the point-wise bound:
            \begin{align}\label{temp-2}
                \omega_k^2 \leq C_{R,N} \left(\omega_k^{2+2\epsilon}+ |d\omega_k|^2\omega_k^{2\epsilon}\right) \Leftrightarrow \omega_k^{2\epsilon}\left(1+ |d\log(\omega_k)|^2\right) \geq C_{R,N} \,,
            \end{align}
            for weights $\omega_k$ as in the statement of this proposition. Arguing by contradiction, since $\omega_k$s form a compact family, \cref{temp-2} fails if and only if for some $\{y_j\}_{j=1}^\infty\in B_R$ we have:
            \begin{align*}
                \lim_{j\to\infty} \omega_k^{2\epsilon}(y_j)\left(1+|d\log\omega_k(y_j)|^2\right) = 0\,.
            \end{align*}
            We can see by compactness that there exists some $y\in B_R$ such that $\omega_k^{2\epsilon}(y)(1+|d\log\omega_k(y)|^2)=0$, meaning also that $\omega_k(y)=0$; now since the vanishing order at zeros of $\omega_k^\epsilon$ with at most $N$ roots is less than $\epsilon N < 1$, we can see that $\omega_k^{\epsilon}$ vanishes slower than $|d\log(\omega_k)| \sim O(|x-|^{-1})$; hence \cref{temp-2} follows. Then we can estimate:
            \begin{align*}
                1 = \int_{B_R} \omega_k^2|h_k|^2 \leq C_{R,N} \int_{B_R} \left(\omega_k^{2+2\epsilon} + |d\omega_k|^2\omega_k^{2\epsilon}\right) |h_k|^2 \leq C_{R,N} \int_{B_R} \omega_k^{2+2\epsilon}|dh_k|^2\,.
            \end{align*}
            In the last inequality we used \cref{Generalized-CKN-thm} with the weight $\omega_k^{1+\epsilon}$.

            In the rest of the proof we fix
            \begin{align*}
                \epsilon = \frac{1}{2N}\,.
            \end{align*}
            \fbox{\textit{Step 3}.} In this step we show that $\omega^{1+\epsilon}v_k$ is also lower bounded in $L^2$. We use \cref{halavati-1} to see:
            \begin{align}\label{last-temp-1}
            \begin{aligned}
                \int_{B_R} \omega_k^{2+2\epsilon}|dh_k|^2 &\leq C_{R,N}\left(\delta_k^2 + \int_{B_R}\omega_k^{2+2\epsilon}|dv_k|^2\right) \\ &\leq C_{R,N}\left(\frac{\delta_k^2}{\epsilon^2} + \int_{B_R}\omega_k^{2+2\epsilon}|dp_k|^2\right)\,.
            \end{aligned}
            \end{align}
            Now by standard elliptic estimates we have that for $p\in W^{2,2}_0(B_R)$:
            \begin{align}\label{last-temp-2}
            \begin{aligned}
                \int_{B_R} \omega_k^{2+2\epsilon}|dp_k|^2 &\leq C_{R,N} \int_{B_R}|dp_k|^2 \\ \leq C_{R,N} \int_{B_R}|\Delta p_k|^2 &\leq C_{R,N} \left(\delta_k^2 + \int_{B_R}V_k^2|h_k|^2\right)\,.
            \end{aligned}
            \end{align}
            Then by the point-wise bound $V(x) \leq \Lambda \omega_k^{1+\frac1N}$ we get that:
            \begin{align*}
                \int_{B_R}V_k^2|h_k|^2 \leq C_{R,N,\Lambda} \int_{B_R} \omega_k^{2+2\epsilon} |h_k|^2 \leq C_{R,N,\Lambda} \int_{B_R} \omega_k^{2+2\epsilon} \left(|v_k|^2 + |v_k+h_k|^2\right)\,.
            \end{align*}
            By Poincar\'e inequality and \cref{Generalized-CKN-thm} with the weight $\omega_k^{1+\epsilon}$, we get that:
            \begin{align*}
                \int_{B_R} \omega_k^{2+2\epsilon} |v_k+h_k|^2 \leq C_R \int_{B_R}\left|d\left(\omega_k^{1+\epsilon}(v_k+h_k)\right)\right|^2 \leq C_{R,N,\Lambda} \delta_k^2\,.
            \end{align*}
            Noting that $\epsilon = \frac{1}{2N}$, we get the following lower bound:
            \begin{align}\label{hodge-bound-1}
                C_{R,N,\Lambda} \leq \int_{B_R} \omega_k^{2+2\epsilon}|v_k|^2\,,
            \end{align}
            provided that $k$ is large enough.
            
            \fbox{\textit{Step 4}.} In this step we find uniform upper bounds on the $W^{1,2}(B_R)$ norm of $\omega^{1+\frac{1}{2N}}v_k$. For this we apply \cref{Generalized-CKN-thm} with $\omega_k^{1+\epsilon}$, the inequality \cref{halavati-1} and standard elliptic estimates as follows:
            \begin{align*}
                \|\omega_k^{1+\epsilon}v_k\|_{W^{1,2}_0(B_R)}^2 &\leq C_{R} \int_{B_R} |\omega_k|^{2+2\epsilon} |dv_k|^2 \leq C_{R,N}\left(\delta_k^2 + \int_{B_R} \omega_k^{2+2\epsilon}|dp_k|^2\right) \\ &\leq  C_{R,N}\left(\delta_k^2 + \int_{B_R} |\Delta p_k|^2\right) \leq C_{R,N}\left(\delta_k^2 + \int_{B_R} V_k^2|h_k|^2\right)\\ &\leq C_{\Lambda,R,N}\left(\delta_k^2 + \int_{B_R} \omega_k^2 |h_k|^2\right) \leq C_{\Lambda,R,N}\,.
            \end{align*}
            We also bound the weighted Sobolev norms of $h_k,p_k$ and $f_k$:
            \begin{align*}
                \|\omega_k^{-1} df_k \|_{L^2(B_R)}^2\leq \delta_k^2\text{ , }\| p_k \|_{W^{2,2}_0(B_R)}^2 \leq C_{R,N,\Lambda} \text{ and } \|\omega_{k} h_k\|_{L^2(B_R)}^2 \leq C_{R,N,\Lambda}\,.
            \end{align*}
            Using \cref{Generalized-CKN-thm} with $\omega_k$ we also estimate $\omega_kv_k$:
            \begin{align*}
                \|\omega_k v_k\|_{L^2(B_R)}^2 &\leq 2\int_{B_R} \omega_k^2\left(|h_k|^2+|v_k+h_k|^2\right) \leq C_R\left(1 + \int_{B_R} \left|d\left(\omega_k\left(h_k + v_k\right)\right)\right|^2\right)\\
                & \leq C_{R}\left(1+\int_{B_R}\omega_k^2 |dh_k+dv_k|^2\right) \leq C_R\left(1+\delta_k^2\right) \leq C_R\,.
            \end{align*}
            \fbox{\textit{Step 5}.} In this last step we finish the compactness argument. By the compact embedding $W^{1,2}_0(B_R) \hookrightarrow_c L^2(B_R)$, Banach–Alaoglu and Rellich–Kondrachov theorem we can extract a sub-sequence $k_j$ with some $g_\infty,f_\infty,p_\infty,\omega_\infty$ such that:
            \begin{align*}
                \begin{cases}
                \omega_{k_j}^{1+\epsilon} v_{k_j} \rightarrow \omega_{\infty}^{1+\epsilon}v_\infty &\text{ strongly in }L^2(B_R)\,,\\
                p_{k_j} \rightarrow p_\infty &\text{ strongly in }W^{1,2}_0(B_R)\,,\\
                \omega_{k_j} v_{k_j} \weakto \omega_{\infty}v_\infty &\text{ weakly in }L^2(B_R)\,,\\
                \omega_{k_j}^{-1} df_{k_j} \weakto \omega^{-1}_{\infty} df_\infty &\text{ weakly in }L^2(B_R)\,,\\
                \omega_{k_j}h_{k_j} \weakto \omega_{\infty}h_\infty &\text{ weakly in }L^2(B_R)\,,\\
                \omega_{k_j}\rightarrow \omega_\infty &\text{ in } C^k(B_R) \text{ for all } k \geq 0\,,\\
                V_{k_j} \weakto V_\infty &\text{ weakly, in duality with } L^1(B_R) \,.
                \end{cases}
                \end{align*}
                Since $\delta_{k_j}\rightarrow 0$ vanishes in the limit. By lower semi-continuity and \cref{halavati-1}
                \begin{align*}
                        f_\infty = 0\text{ and } p_\infty = g_\infty \Rightarrow -\Delta v_\infty + V_\infty(x)v_\infty = 0 \text{ in the sense of distributions}\,.
                \end{align*}
                Finally we test by $v_\infty$ and use $V_\infty(x)\geq 0$ to get that $d v_\infty = 0$. Since $v$ vanishes on the boundary then $v_\infty = 0$, but this contradicts \cref{hodge-bound-1} which concludes the proof.
            \end{proof}
        \end{proposition}
        \begin{corollary}\label{compactness-inequality-2}
            As a consequence of the estimates in \cref{compactness-inequality-1} we also get that there exists a constant $C_{\Lambda,R,N}$ such that for any $0<\epsilon<\frac1N$:
            \begin{align*}
                \int_{B_R} \omega^{2+2\epsilon} |dh|^2 \leq \frac{C_{\Lambda,R,N}}{\epsilon^2} \int_{B_R} \omega^2|\star dh + B|^2 + |\star dB + V(x)h|^2\,.
            \end{align*}
            \begin{proof}
                The above inequality follows immediately from applying the conclusion of \cref{compactness-inequality-1} to the estimate \cref{last-temp-1} and \cref{last-temp-2}.
            \end{proof}
        \end{corollary}
        
        \subsection{Combining near and far estimates}\label{combining-near-far-subsection} In this section we combine local weighted estimates near the vortex set and uniform elliptic estimates far from the vortex set with a concentration compactness type argument. Roughly speaking, we show that the \textit{discrepancy} cannot concentrate in an intermediate annulus around the vortex set. This allows us to \textit{glue} the near and far estimates:
        \begin{proof}[Proof of \cref{regular-stability}]
            Using the first assumption we can gauge fix $(u,\nabla)$ in such a way that $u$ and $u_0$ have the same phase almost everywhere; precisely:
            \begin{align*}
                (u,\nabla) \rightarrow (u\gamma, \nabla - i\gamma^*(d\theta))\text{ where } \gamma = \frac{u_0}{|u_0|}(\frac{u}{|u|})^{-1}\,.
            \end{align*}
            Note that $ d(\gamma^*(d\theta)) =  d((\frac{u_0}{|u_0|})^* (d\theta)) -  d((\frac{u}{|u|})^* (d\theta)) = 0$ since $u$ and $u_0$ have the same zero set, counted with multiplicity. Then we can write the discrepancy as:
            \begin{align*}
                E(re^{i\theta},A) - 2\pi N 
                &= \int_{\R^2} r^2|\star d\log(r) + A - d\theta|^2 + |\star dA - \frac{1-r^2}{2}|^2 \\
                &= \int_{\R^2} r^2|\star dh + B|^2 + |\star dB + r_0^2 \frac{e^{2h} - 1}{2}|^2\,,
            \end{align*}
            where $h = \log\left(\frac{r}{r_0}\right)$ and $B = A-A_0$. By the third assumption we have the bound $\|h\|_{L^\infty(\R^2)} \leq \log(\Lambda)$ which also means $\frac{e^{2h}-1}{2}$ is comparable with $h$: $$\left(\frac{1-\Lambda^{-2}}{2\log\Lambda}\right)  \leq \left(\frac{e^{2h}-1}{2h}\right) \leq \left(\frac{\Lambda^2-1}{2\log\Lambda}\right) \,,$$ so we can write:
            \begin{align*}
                E(re^{i\theta},A) - 2\pi N \geq \int_{\R^2} \Lambda^{-2} r_0^2|\star dh + B|^2 + |\star dB + V(x) h|^2\,,
            \end{align*}
            for some positive potential $\left(\frac{1-\Lambda^{-2}}{2\log\Lambda}\right)r_0^2(x)  \leq V(x) \leq \left(\frac{\Lambda^2-1}{2\log\Lambda}\right)r_0^2(x)$. Now by \cref{vortex-set-estimates} we know that $r_0$ behaves like the weights defined in \cref{compactness-inequality-1}; since $h$ is not necessarily compactly supported, in order to apply the estimates \cref{compactness-inequality-1}, we use a concentration compactness type argument. Our goal is to prove that there exists a constant $C_{\Lambda,N}>0$ such that:
            \begin{align}
                \int_{\R^2} r_0^2|h|^2 \leq C_{\Lambda,N}\int_{\R^2} r_0^2 |\star dh + B|^2 + |\star dB + V(x)h|^2\label{target-bound-1}\,.
            \end{align}
            Arguing by contradiction, there exists a sequence $\{r_{0^k},h_k,B_k,V_k\}_{k=1}^\infty$ such that:
            \begin{align*}
                \int_{\R^2} r_{0^k}^2|h_k|^2 &= 1\,,\\
                \eta_k^2 = \int_{\R^2} r_{0^k}^2|\star dh_k + B_k|^2 + |\star d&B_k + V_k(x) h_k|^2 \rightarrow 0 \text{ as } k\rightarrow \infty\,.
            \end{align*}
            Now by \cref{vortex-set-estimates} (applied for each $k$) there exists $M_k$ balls $B_{\rho_{k,1}}(z_{k,1}) , \dots , B_{\rho_{k,M_k}}(z_{M_k})$ for some $M_k \leq N$ such that $C_N^{-1} \leq \rho_{k,j} \leq C_N$ and $\{r_{0^k} \leq \beta_k \} \subset \cup_{j=1}^{M_k} B_{\rho_{k,j}}(z_{k,j})$ for some $\frac14 < \beta_k < \frac12$ and $B_{2\rho_{k,i}}(z_{k,i}) \cap B_{2\rho_{k,j}}(z_{k,j}) = \emptyset$ for all $1 \leq i < j \leq M_k$. Now take $\phi_k,\psi_k$ and $\chi_k$ to be three smooth cut-off functions on $\R^2$ defined as follows:
            \begin{align*}
                &\begin{cases}
                    \phi_k = 0 &\text{on } \R^2 \backslash \bigcup_{j=1}^{M_k} B_{2\rho_{k,j}}(z_{k,j})\,,\\
                    \phi_k = 1 &\text{on } \bigcup_{j=1}^{M_k} B_{(1+\frac23)\rho_{k,j}}(z_{k,j})\,,
                \end{cases}\\
                &\begin{cases}
                    \psi_k = 0 &\text{on } \bigcup_{j=1}^{M_k} B_{(1+\frac13)\rho_{k,j}}(z_{k,j})\,,\\
                    \psi_k = 1 &\text{on } \R^2 \backslash \bigcup_{j=1}^{M_k} B_{(1+\frac23)\rho_{k,j}}(z_{k,j})\,,
                \end{cases}
                \\
                &\begin{cases}
                    \chi_k = 0 &\text{on } \R^2 \backslash \bigcup_{j=1}^{M_k} \big( B_{2\rho_{k,j}}(z_{k,j}) \backslash B_{\rho_{k,j}}(z_{k,j}) \big)\,,\\
                    \chi_k = 1 &\text{on } \bigcup_{j=1}^{M_k} \big( B_{(1+\frac23)\rho_{k,j}}(z_{k,j}) \backslash B_{(1+\frac13)\rho_{k,j}}(z_{k,j}) \big)\,,
                \end{cases}
            \end{align*}
             with the point-wise estimate $|d \phi_k| + |d \psi_k| + |d \chi_k| \leq C_N$. Note that this is possible since $C_N^{-1} \leq \rho_{k,j} \leq C_N$. Then there exists $e_1 \in L^2(\R^2;\R^2)$ and $e_2 \in L^2(\R^2;\R)$ such that:
            \begin{align}
                -\Delta h_k + V_k(x) h_k &= -\div(\frac{e_1}{r_{0^k}}) + e_2\label{PDE-2}\,,\\
                \|e_1\|_{L^2(\R^2)}^2 &+ \|e_2\|_{L^2(\R^2)}^2 = \eta_k^2\nonumber\,.
            \end{align}
            We test \cref{PDE-2} with $\psi_k h_k$, apply Young's inequality and use the estimates on $r_{0^k}$ in \cref{vortex-set-estimates} inside the annulus $B_{2\rho_{k,j}} \backslash B_{\rho_{k,j}} (z_{k,j})$ to see that
            \begin{align}
                \int_{\mathcal{B}_k} |h_k|^2 + |dh_k|^2 &\leq C_N(\eta_k^2 + \int_{\mathcal{A}_k}|h_k|^2)\label{annulus-bound-1}\,,\\
                \text{where } \mathcal{B}_k = \R^2 \backslash \bigcup_{j=1}^{M_k} B_{(1+\frac23)\rho_{k,j}}(z_{k,j})\;&\text{ and }\;
                \mathcal{A}_k = \bigcup_{j=1}^{M_k} \big( B_{(1+\frac23)\rho_{k,j}}\backslash B_{(1+\frac13)\rho_{k,j}}(z_{k,j}) \big)\nonumber\,.
            \end{align}
            Now we apply \cref{compactness-inequality-1} and \cref{vortex-set-estimates} to $\phi_k h_k,\phi_k B_k$, together with $C_\Lambda^{-1} r_{0^k}^2 \leq V_k \leq C_\Lambda r_{0^k}^2$; then we collect the boundary terms using the bound $|d\phi_k| \leq C_N$ to see that:
            \begin{align*}
                \int_{\mathcal{B}_k^c} r_{0^k}^2|h_k|^2 &\leq C_{\Lambda,N} \int_{\R^2} r_{0^k}^2|\star d(\phi_k h_k) + (\phi_k B_k) |^2 + |\star d(\phi_k B_k) + V_k(x)(\phi_k h_k)|^2\\
                &\leq C_{\Lambda,N} (\eta_k^2 + \int_{\mathcal{B}_k}|h_k|^2 + |dh_k|^2) \leq_{\cref{annulus-bound-1}} C_{\Lambda,N}(\eta_k^2 + \int_{\mathcal{A}_k} |h_k|^2)\,.
            \end{align*}
            Then we add the estimate above to \cref{annulus-bound-1} to see the lower bound:
            \begin{align*}
                1=\int_{\R^2} r_{0,k}^2|h_k|^2 \leq C_{\Lambda,N}(\eta^2_k + \int_{\mathcal{A}_k} |h_k|^2)\,.
            \end{align*}
            So we get that for large enough $k$, for a possibly different constant:
            \begin{align*}
                \int_{\mathcal{A}_k} |h_k|^2 \geq C_{\Lambda,N}>0\,.
            \end{align*}
            Finally testing \cref{PDE-2} with $\chi_k h_k$ we get that:
            \begin{align}
                \int_{\mathcal{A}_k} |h_k|^2 + |dh_k|^2 \leq C_{\Lambda,N}\label{bound-1}\,.
            \end{align}
            Notice that by definition $\mathcal{A}_k = \bigcup_{j=1}^{M_k} \big( B_{(1+\frac23)\rho_{k,j}}\backslash B_{(1+\frac13)\rho_{k,j}}(z_{k,j}) \big)$ is a disjoint union of $M_k \leq N$ annuli so we get that there is at least one annulus in which the energy is concentrating:
            \begin{align}
                \int_{\mathcal{A}_{k,0}} |h_k|^2 \geq C_{\Lambda,N}\text{ where }
                \hat{\mathcal{A}_{k}} = B_{(1+\frac23)\rho_{k,j_0}}\backslash B_{(1+\frac13)\rho_{k,j_0}}(z_{k,j_0})\label{bound-2}\,.
            \end{align}
            Then we use \cref{compactness-inequality-2}, \cref{Generalized-CKN-thm} and similar computations in the preceding paragraphs to get uniform bounds as follows:
            \begin{align}
                \|r_{0^k}^{1+\frac{1}{2N}} h_k\|_{W^{1,2}(\R^2)}^2 \leq C_{\Lambda,N}\label{bound-3}\,.
            \end{align}
            Note that if the right hand side of \cref{target-bound-1} is zero we get that $-\Delta h + V(x) h = 0 $. Testing with $h$ and integrating by parts we get that $h=0$. From the uniform bounds \cref{bound-1,bound-3} we get that there exists a sub-sequence (not relabeled) such that:
            \begin{align*}
                \tilde r_{0^k}^{1+\frac{1}{2N}}&\tilde{h}_k \weakto 0 \text{ weakly in } H^1(\R^2)\,,\\
                &\tilde{h}_k \rightarrow 0 \text{ strongly in } L^2(B_{1+\frac23}\backslash B_{1+\frac13})\,,\\
                \text{where } \tilde r_{0^k}^{1+\frac{1}{2N}}&\tilde{h}_k(x) = r_{0^k}^{1+\frac{1}{2N}}h_k(\rho_{k,j_0} x + z_{k,j_0})\,.
            \end{align*}
            This contradicts the lower bound \cref{bound-2} on the annulus $\hat{\mathcal{A}_{k}}$ and we conclude.
        \end{proof}
        \begin{proof}[Proof of \cref{main-result-epsilon-sobolev}]
                We apply \cref{compactness-inequality-2} combined with the estimate \cref{annulus-bound-1} on each annulus and conclude the proof.
            \end{proof}

    \subsection{Stability of the Jacobian}\label{jacobian-regular-stability-section}
    Recall the definition of the Yang-Mills-Higgs Jacobian:
        \begin{align*}
            J(u,\nabla) = \Psi(u) + \omega(1-|u|^2)\text{ where } \Psi(u)(j,k) = 2\ang{\nabla_{\de_j} u ,i\nabla_{\de_k}u}\,,
        \end{align*}
        for $1\le j,k\le 2$. Here $\omega$ is the real curvature two-form associated to $F_\nabla$. Note that the definition of the Yang-Mills-Higgs Jacobian is gauge invariant. We recall the statement of \cref{jacobian-main-result}:
        \begin{thm*}
        For any $N\in\mathbb{N}$ and $\Lambda >1$ there exists $C_{\Lambda,N},\eta_{\Lambda,N}>0$ with the following properties. Let $(u,\nabla)\in W^{1,2}_{loc}(\R^2)$ be a section and connection on the trivial line bundle $\R^2\times \C$ such that
        \begin{enumerate}[label=(\roman*)]
            \item $\star d\left( (\frac{u}{|u|})^*(d\theta)\right) = 2\pi\sum_{k=1}^{N}\delta_{x_k}$ for a collection of points $\{x_k\}_{k=1}^{N} \subset \R^2$ (counted with multiplicity).
            \item $E(u,\nabla) - 2\pi N \leq \eta^2_{\Lambda,N}$\,.
            \item $\Lambda^{-1}|u_0| \leq |u| \leq \Lambda |u_0|$ for some $N$-vortex solution $(u_0,\nabla_0)$ with $\{x_k\}_{k=1}^{N}$ as the zero set (counted with multiplicity)\,.
        \end{enumerate}
        Then we have the following estimate:
        \begin{align*}
            \int_{\R^2} \left|J(u,\nabla) - J(u_0,\nabla_0)\right| \leq C_{\Lambda,|N|} \sqrt{E(u,\nabla)-2\pi |N|}\,.
        \end{align*}
        \begin{proof}
            Recall the estimate of \cref{main-result-epsilon-sobolev} for any $0<\epsilon < \frac1N$
            \begin{align*}
                \int_{\R^2}|u_0|^{2+2\epsilon}\left[\left|A-A_0\right|^2 + \left|d\log\left(\frac{|u|}{|u_0|}\right)\right|^2\right]\leq \frac{C_{\Lambda,|N|}}{\epsilon^2}\left[E(u,\nabla)-2\pi|N|\right]\,,
            \end{align*}
            Note that we can write the Jacobian in a gauge invariant way. Using \textit{(i)} we can gauge fix as in the proof of \cref{regular-stability} such that $u = re^{i\theta}$ and $u_0 = r_0e^{i\theta}$ have the same phase. Then we can rewrite the Yang-Mills-Higgs Jacobian as follows:
            \begin{align*}
                J(u,\nabla) = (1-r^2)dA -2rdr\wedge(A-d\theta)\,.
            \end{align*}
            To estimate the difference, we see that:
            \begin{align*}
                \int_{\R^2}\left|J(u,\nabla) - J(u_0,\nabla_0)\right| \leq &\int_{\R^2}\left|(1-r^2)dA - (1-r_0^2)dA_0\right| \\ + &\int_{\R^2}\left|2rdr\wedge (A-d\theta) - 2r_0dr_0\wedge(A_0-d\theta)\right| = \textbf{I} + \textbf{II}\,.
            \end{align*}
            We use \cref{regular-stability} and estimate the first term:
            \begin{align*}
                \textbf{I} &\leq \int_{\R^2} \left|(1-r^2)dA - (1-r_0^2)dA_0\right| \\ &\leq \left[\int_{\R^2}|r^2-r_0^2|^2\right]^{\frac12}\left[\int_{\R^2}|dA|^2\right]^{\frac12} + \left[\int_{\R^2}(1-r_0^2)^2\right]^{\frac12}\left[\int_{\R^2}|dA-dA_0|^2\right]^{\frac12} \\ &\leq C_{\Lambda,|N|}\sqrt{E(u,\nabla)-2\pi|N|}\,.
            \end{align*}
            Then for the second term we proceed as follows:
            \begin{align}
            \begin{aligned}\label{jacobian-temp-1}
                \textbf{II} \leq &\int_{\R^2} \left|rdr\wedge (A-d\theta) - r_0dr_0\wedge (A_0-d\theta)\right| \\ \leq &C_{\Lambda}\left[\int_{\R^2}r_0^{2\epsilon-2}|rdr-r_0dr_0|^2\right]^{\frac12}\left[\int_{\R^2}r_0^{2-2\epsilon}|A_0-d\theta|^2\right]^{\frac12} \\ + &C_\Lambda\left[\int_{\R^2} r_0^{2+2\epsilon}|A-A_0|^2\right]^{\frac12}\left[\int_{\R^2}r_0^{-2\epsilon}|dr_0|^2\right]^{\frac12}\,.
                \end{aligned}
            \end{align}
            Take $\frac{1}{3N}<\epsilon<\frac1N$ and estimate as follows:
            \begin{align}
            \begin{aligned}\label{jacobian-temp-2}
                \int_{\R^2}& r_0^{2\epsilon-2}|rdr -r_0dr_0|^2 = \\ &=\int_{\R^2} r_0^{2\epsilon-2}\left|r^2d\log(r)-r_0^2d\log(r_0)\right|^2  \\
                &\leq C_{\Lambda}\int_{\R^2} r_0^{2+2\epsilon}\left|d\log \left(\frac{r}{r_0}\right)\right|^2 + r_0^{2\epsilon-2}|d\log(r_0)|^2(r^2-r_0^2)^2\\
                &\leq C_{\Lambda,|N|}\left[E(u,\nabla)-2\pi|N| \right] + C_{\Lambda}\int_{\R^2} r_0^{2\epsilon+2}|d\log(r_0)|^2\left|\log\left(\frac{r}{r_0}\right)\right|^2\,.
                \end{aligned}
            \end{align}
            In the last line we used assumption \textit{(iii)} to see that $$ C_\Lambda^{-1}|\log\left(\frac{r}{r_0}\right)| \leq |\frac{r}{r_0}-1| \leq C_\Lambda|\log\left(\frac{r}{r_0}\right)|\,.$$
            Now let $\{B_{\rho_k}(z_k)\}_{k=1}^{M}$ be the covering constructed in \cref{vortex-set-estimates} and let $\omega_k$ the associated weights. Then locally we have:
            $$\|d\log\left(r_0\right) - d\log(\omega_k)\|_{L^{\infty}(B_{2\rho_k}(z_k))} \leq C_N\,.$$
            Moreover take smooth indicators $\phi_k$ for $B_{\rho_k}(z_k)$ such that $\phi =1 $ on $B_{\rho_k}(z_k)$ and zero outside $B_{2\rho_k}(z_k)$ with $|d\phi_k| \leq C_N$; we see that the last term in \cref{jacobian-temp-2} is bounded as follows:
            \begin{align*}
                \int_{\R^2} r_0^{2\epsilon+2}|d\log(r_0)|^2\left|\log\left(\frac{r}{r_0}\right)\right|^2 \leq 
                C_{N,\Lambda}\sum_{k=1}^{M}&\int_{B_{2\rho_k}(z_k)} \phi_k^2\omega_k^{2\epsilon+2}(1+|d\log(\omega_k)|^2)\left|\log\left(\frac{r}{r_0}\right)\right|^2\\
                +C_{N,\Lambda}&\int_{\R^2\backslash \bigcup_{k=1}^{M}B_{\rho_k}(z_k)} |\log\left(\frac{r}{r_0}\right)|^2
            \end{align*}
            Then we use \cref{Generalized-CKN-thm} with $\omega_k^{2+2\epsilon}$ and Poincar\'e inequality for the first sum in the above display to see that:
            \begin{align*}
                 \sum_{k=1}^{M}&\int_{B_{2\rho_k}(z_k)} \phi_k^2\omega_k^{2\epsilon+2}(C+|d\log(\omega_k)|^2)\left|\log\left(\frac{r}{r_0}\right)\right|^2 \\ \leq C_{N,\Lambda}\sum_{k=1}^{M}&\int_{B_{2\rho_k}(z_k)} \left|\phi_k\omega_k^{1+\epsilon}\log\left(\frac{r}{r_0}\right)\right|^2 + \phi_k^2|d(\omega_k^{1+\epsilon})|^2\left|\log\left(\frac{r}{r_0}\right)\right|^2\\
                 \leq C_{N,\Lambda}\sum_{k=1}^{M}&\int_{B_{2\rho_k}(z_k)} \phi_k^2\omega_k^{2+2\epsilon}\left|d\log\left(\frac{r}{r_0}\right)\right|^2 \leq C_{N,\Lambda}\left[E(u,\nabla)-2\pi|N|\right]\,.
            \end{align*}
            Here the last line follows by \cref{compactness-inequality-2} (or \cref{main-result-epsilon-sobolev}). We then use \cref{annulus-bound-1} to bound $\log(r/r_0)$ on $\R^2\backslash \bigcup_{k=1}^{M}B_{\rho_k}(z_k)$ and see that:
            \begin{align*}
                \int_{\R^2} r_0^{2\epsilon+2}|d\log(r_0)|^2\left|\log\left(\frac{r}{r_0}\right)\right|^2 \leq C_{N,\Lambda}\left[E(u,\nabla)-2\pi|N|\right]\,.
            \end{align*}
            Using the above display we can bound \cref{jacobian-temp-2} as follows:
            \begin{align*}
                \int_{\R^2} r_0^{2\epsilon-2}|rdr -r_0dr_0|^2 \leq C_{\Lambda,N}\left[E(u,\nabla)-2\pi|N|\right]\,.
            \end{align*}
            Now all that remains is to bound the second term of \cref{jacobian-temp-2}. Notice that if $\epsilon < \frac{1}{N}$ the term $r_0^{-2\epsilon}$ is comparable to $|x-p|^{2-\delta}$ around any vortex point $p$ for some $\delta>0$. This means that it is locally integrable, hence
            \begin{align*}
                \int_{\R^2} r_0^{-2\epsilon} |dr_0|^2 < C_{N,\epsilon}\,.
            \end{align*}
            Here we used that the integral of $|dr_0|^2$ is bounded on the whole domain. Now we use the estimates on the connection in \cref{main-result-epsilon-sobolev} to see that for $\epsilon=\frac{1}{2N}$:
            \begin{align*}
                \textbf{II} \leq C_{\Lambda,N}\sqrt{E(u,\nabla)-2\pi|N|}\,.
            \end{align*}
            This is indeed the desired conclusion.
        \end{proof}
        \end{thm*}
        
    \section{A selection principle and the proof of stability}\label{selection-principle-section}
    The main goal of this section is to drop the first and the third assumption in \cref{regular-stability}. We prove that for any $(u,\nabla)\in W^{1,2}_{loc}(\R^2)$ that nearly minimizes the energy, we can \textit{select} another pair $(\tilde u,\tilde\nabla)$ close enough to $(u,\nabla)$ that satisfies the assumptions of \cref{regular-stability}. We do this by inductively replacing $(u,\nabla)$ with a minima of the penalized functional \cref{auxilary-energy} (proof of existence in \cref{auxilary-energy-existence}). In each iteration while staying close to the original $(u,\nabla)$ we gain at least one derivative of regularity (proof of regularity in \cref{auxilary-energy-regularity}). Then using \cref{complex-polynomial-perturbration} (smooth perturbation of complex polynomials) we show that after finitely many steps the new pair $(\tilde u,\tilde \nabla)$ satisfies the assumptions of \cref{regular-stability}.
    
    \begin{thm}\label{selection-principle-thm}
        There exists constants $C_N,\Lambda_N>0$ with the property that for any $N$-vortex section and connection $(u,\nabla)$ with finite energy $E(u,\nabla) = 2\pi |N| + \eta^2$ there exists another $N$-vortex section and connection $( \tilde u,\tilde \nabla)$ such that:
        \begin{enumerate}[label=(\roman*)]
            \item $\|u-\tilde u\|^2_{L^2(\R^2)} + \|F_{\nabla} - F_{\tilde\nabla}\|^2_{L^2(\R^2)} \leq C_N \eta^2$,
            \item $\Lambda_N^{-1} |u_0|\leq | \tilde u| \leq \Lambda_N|u_0|$ for some $N$-vortex solution $(u_0,\nabla_0)$ to the vortex equations \cref{vortex-eq},
            \item $2\pi N \leq E(\tilde u,\tilde \nabla) \leq E(u,\nabla)$,
        \end{enumerate}
        provided that $\eta$ is small enough.
    \end{thm}
    
    To prove \cref{selection-principle-thm} we use a penalized energy to find a \textit{selection principle}:
    \begin{align}\label{auxilary-energy}
        \mathcal{G}_{(u,\nabla)}(u_1,\nabla_1)=E(u_1,\nabla_1) + \|u_1-u\|_{L^2(\R^2)}^2 + \|A_1 -A\|_{L^2(\R^2)}^2\,.
    \end{align}
    This energy also enjoys the coupled gauge invariance
    \begin{align*}
        (u,\nabla) \rightarrow (ue^{i\xi},A+d\xi)\text{ and }
        (u_1,\nabla_1) \rightarrow(u_1e^{i\xi},A_1+d\xi)\,,
    \end{align*}
    for any smooth compactly supported function $\xi \in C^\infty_c(\R^2)$. Note that both pairs have to be gauge transformed with the same function.
    
    \subsection{Existence} We first prove the existence of a minimizer of \cref{auxilary-energy} via the direct method of the calculus of variations:
    \begin{lemma}\label{auxilary-energy-existence}
        For any $N$-vortex section and connection $(u,\nabla)$ with $E(u,\nabla) = 2\pi N + \eta^2$, the penalized energy $\mathcal{G}_{(u,\nabla)}$ \cref{auxilary-energy} achieves its minimum for some $N$-vortex pair $(u_1,\nabla_1)\in W^{1,2}_{loc}(\R^2)$.
        \begin{proof}
        We have the lower bound:
        \begin{align*}
            \mathcal{G}_{(u,\nabla)}(u_1,\nabla_1) \geq E(u_1,\nabla_1) \geq 2\pi N\,.
        \end{align*}
        Hence there exists an $N$-vortex sequence $(v_{j},\nabla_{j})$ realizing the infimum of \cref{auxilary-energy}:
        \begin{align*}
            \lim_{j\rightarrow\infty} \mathcal{G}_{(u,\nabla)}(v_{j},\nabla_{j}) = \inf_{(u_1,\nabla_1) \in W^{1,2}_{loc}} \mathcal{G}_{(u,\nabla)}(u_{1},\nabla_{1})\,.
        \end{align*}
        Let $\Omega$ be any bounded smooth simply connected domain. By the gauge invariance of \cref{auxilary-energy}, we can gauge fix $(v_{j} ,\nabla_{j}) \rightarrow (w_{j},B_{j})$ in the Coulomb gauge such that $B_j$ is divergence free $d^{*}B_{j} = 0$ inside the domain $\Omega$ and its normal component vanishes $\iota_\nu B_{j} = 0$ on the boundary $\de\Omega$. By Gaffney inequalities (see \cite[Theorem 4.8]{nonlinear-hodge}) and a compactness argument we can bound:
        \begin{align*}
            \|B_{j}\|_{W^{1,2}(\Omega)}^2 \leq C_\Omega \|dB_{j}\|_{L^2(\Omega)}^2 \leq C_{\Omega,N}\,.
        \end{align*}
        By the global Sobolev embedding $W^{1,2}(\R^2) \hookrightarrow_c L^4(\R^2)$ we get the following bounds:
        \begin{align*}
            \|w_{j}\|^2_{W^{1,2}(\Omega)} &= \int_{\Omega}|w_j|^2+|dw_j|^2 \\ &\leq C_{\Omega}\left(1 + \sqrt{\int_{\Omega} (1-|w_j|^2)^2}\right) +C \int_{\Omega} |dw_j - iw_jB_j|^2 + |w_jB_j|^2 \\&\leq C_{\Omega,N} + C\|w_{j} B_{j}\|^2_{L^2(\Omega)} \leq C_{\Omega,N} + C\|w_{j}\|_{L^4(\Omega)}^2 \|B_{j}\|^2_{L^4(\Omega)}\\
            &\leq  C_{\Omega,N} + C\|w_{j}\|_{W^{1,2}(\Omega)}^2 \|B_{j}\|^2_{W^{1,2}(\Omega)} \leq C_{\Omega,N}\,.
        \end{align*}
        Then we can find a sub-sequence $(w_{j},B_{j})$ (not necessarily relabeled) and a limit section and connection $(u_1,\nabla_1) \in W^{1,2}_{loc}$ such that after gauge fixing $(u_{1},A_{1})$ in the Coulomb gauge in $\Omega$ we get that:
        \begin{align*}
            (w_{j},B_{j}) &\weakto (u_{1} , A_{1}) \text{ weakly in } W^{1,2}_{loc}(\R^2)\,.
        \end{align*}
        We need to show that $(u_1,\nabla_1)$ is also an $N$-vortex pair. First observe that $(u,\nabla)$ is a competitor with energy $\mathcal{G}_{(u,\nabla)}(u,\nabla) = 2\pi N + \eta^2$; by lower semi-continuity we get that: 
        $$\|u_1-u\|^2_{L^2(\R^2)}\leq \liminf_{j\to\infty}\|v_j-u\|_{L^2(\R^2)}^2 \leq \eta^2\,.$$ In particular the difference is bounded in $L^2$. Now consider a smooth kernel $\phi$ and let the $\epsilon$-rescaled version to be $\phi_\epsilon = \frac{1}{\epsilon}\phi(\frac{x}{\epsilon})$. Then consider the mollified functions $u*\phi_\epsilon$ and $u_1*\phi_\epsilon$. By the embedding $W^{1,2}_{loc}(\R^2)\rightarrow\mathrm{VMO}_{loc}(\R^2)$, where $\mathrm{VMO}_{loc}$ is the space of functions with locally vanishing mean oscillation, we can see that:
        \begin{align*}
            u*\phi_\epsilon \rightarrow u \text{ and }u_1*\phi_\epsilon\rightarrow u_1\text{ in } \mathrm{BMO}\cap L^1_{loc}(B_R)\text{ as } \epsilon\to 0\,.
        \end{align*}
        By \cite[Property 2 in Chapter II.2]{Brezis-Nirenberg-degree} we can see that for small enough $\epsilon$ the degree of $u*\phi_\epsilon$ and $u_1*\phi_\epsilon$ is equal to the degree of $u$ and $u_1$, respectively. Now we can estimate for fixed $\epsilon$:
        \begin{align*}
            \lim_{R\rightarrow \infty}\|\phi_\epsilon*(u-u_1)\|_{C^0(\R^2\backslash B_R)} \leq \lim_{R\rightarrow \infty} \epsilon^{-1}\|u-u_1\|_{L^2(\R^2\backslash B_R)} = 0\,.
        \end{align*}
        Hence for all $R$ large enough, the degree of $\phi_\epsilon*u$ coincides with $\phi_\epsilon*u_1$ and we can conclude that $(u_1,\nabla_1)$ is indeed an $N$-vortex pair.
        \end{proof}
    \end{lemma}
    \begin{rmk}\label{bounding-u}
        Without loss of generality we may assume $|u|\leq3$, precisely: For any $(u,\nabla)$ with $\eta>0$ small enough, we can find another section $u_1$ such that $|u_1| \leq 3$ and:
        \begin{align*}
            \|u-u_1\|_{L^2(\R^2)}^2\leq C\eta^2 \text{ and } E(u_1,\nabla) \leq E(u,\nabla)\,,
        \end{align*}
        for some universal constant $C$.
        \begin{proof}
            Identifying $u=re^{i\theta}$ and $\nabla: d-iA$, consider the super level set $\{r \geq 2\}$. From the energy bound we can see that:
            \begin{align*}
                |\{r\geq 2 \}| + \int_{\{r\geq2\}} |dr|^2 \leq C_N\,.
            \end{align*}
            Similar to the proof of \cref{IBP}, using the coarea formula and the mean value theorem we can find $2< \beta < 3$ such that $\hau^1\left(\de\{r\geq \beta\}\right) \leq C_N$. Since perimeter bounds diameter, $\{r\geq \beta\}$ is made of a collection of bounded simply connected domains. Now we unwrap the discrepancy in this domain (with the Bogomolny trick, for more details see \cref{IBP}):
            \begin{align*}
                \eta^2 &\geq \int_{\{r \geq \beta \}} |\star dr + r(A-d\theta)|^2 + |\star dA - \frac{1-r^2}{2}|^2\\
                &= \int_{\{r \geq \beta \}} |dr|^2 + r^2|A-d\theta|^2 + |dA|^2 + \frac{(1-r^2)^2}{4} -d(d\theta)(1-r^2)\,.
            \end{align*}
            Note that $d(d\theta)=0$ away from $\{r=0\}$, hence we can see that:
            \begin{align}\label{super-estimate}
                \int_{\{r \geq \beta \}} \frac{(1-r^2)^2}{4} \leq \eta^2\,.
            \end{align}
            Since $(1-r^2)^2 > 9$ on $\{r\geq \beta\}$ hence \cref{super-estimate} tells us that:
            \begin{align*}
                |\{r \geq \beta\}| \leq \eta^2\,.
            \end{align*}
            Now we define $u_1$:
            \begin{align*}
                u_1 = \begin{cases}
                    u &\text{on } \{|u|<3\}\,,\\
                    3\frac{u}{|u|} &\text{on } \{|u|\geq 3\}\,.
                \end{cases}
            \end{align*}
            We estimate the energy difference:
            \begin{align*}
                E(u,\nabla) - E(u_1,\nabla) = \int_{\{r\geq 3\}} |dr|^2 + (r^2-9)|A-d\theta|^2 +|dA|^2+ \frac{(1-r^2)^2-64}{4} \geq 0\,,
            \end{align*}
            and the difference of sections using \cref{super-estimate}:
            \begin{align*}
                \|u-u_1\|_{L^2(\R^2)}^2 = \int_{\{r\geq3\}}|u_1-u|^2 \leq C|\{r\geq 3\}| + C\int_{\{r\geq 3\}} r^2 \leq C\eta^2\,,
            \end{align*}
            where the last inequality followed from $$\int_{\{r\geq 3\}}r^2 \leq C \int_{\{r\geq 3\}} \frac{(1-r^2)^2}{4} \leq C\eta^2\,.$$
            This indeed gives us the desired conclusion.
        \end{proof}
    \end{rmk}
    \subsection{Regularity} In this section we derive regularity estimates for minimizers of the penalized functional \cref{auxilary-energy}:
    \begin{lemma}\label{auxilary-energy-regularity}
        Let $(u,\nabla)$ be a section and connection with finite energy $E(u,\nabla) = 2\pi N + \eta^2$, with small enough $\eta>0$ and $|u| \leq 3$. Moreover let $(u_1,\nabla_1)$ be a minimizer of \cref{auxilary-energy}. Then for any $0<\alpha<1$ we have the regularity estimates below:
        \begin{enumerate}[label=(\roman*)]
            \item $\|u_1\|_{C^{1,\alpha}(\Omega)} + \|A_1\|_{C^{1,\alpha}(\Omega)} \leq C_{\Omega,N,\alpha}$,
            \item $\|u_1\|_{C^{k+1,\alpha}(\Omega)} + \|A_1\|_{C^{k+1,\alpha}(\Omega)} \leq C_{\Omega,N,k,\alpha} (\| u \|_{C^{k,\alpha}(\Omega)}+\| A \|_{C^{k,\alpha}(\Omega)} + 1)$,
        \end{enumerate}
        for all $k\geq1$ and any domain $\Omega$ where both $(u_1,A_1)$, $(u,A)$ are independently measured in the Coulomb gauge in a slightly bigger domain $U \supset \Omega$.
        \begin{proof}
              We inspire from the strategy in \cite[Appendix]{Pigati-1}. However since we are in two dimensions, the embedding $W^{2,2}(\R^2) \hookrightarrow C^\alpha(\R^2)$ greatly simplifies the proof.

              For any open bounded smooth domain $\Omega$ we take three bigger domains $\Omega \subset \Omega_1 \subset \Omega_2 \subset \Omega_3$. Then we gauge fix $(u_1,A_1)$ in $\Omega_3$ in the Coulomb gauge such that $d^* A_1 = 0$ inside $\Omega_3$ and $A_1(\nu) = 0$ on $\de U$. Now the Euler Lagrange equations for minimizers of \cref{auxilary-energy} are as follows:
            \begin{align}
                \Delta u_1 &= 2\ang{i du_1 , A_1} + |A_1|^2 u_1 - \frac12 (1-|u_1|^2)u_1 + (u_1 - u)\label{auxilary-PDE-1}\,,\\
                \Delta_H A_1 &= \ang{du_1-iu_1A_1,iu_1} + A_1-A\label{auxilary-PDE-2}\,.
            \end{align}
            Here $\Delta_H$ is the Laplace Beltrami operator for one-forms. The Euler Lagrange equations for the difference in gauge is as follows:
            \begin{align}\label{gauge-PDE}
                d^*A = \ang{u,iu_1}\,.
            \end{align}
            By Gaffney type inequalities in \cite{nonlinear-hodge} we estimate:
            \begin{align*}
                \|A_1\|_{W^{1,2}(\Omega_3)} \leq C\|dA_1\|_{L^2(\Omega_3)} \leq C_N\,.
            \end{align*}
            The Euler-Lagrange equation for $|u_1|^2$ is as follows:
            \begin{align}
                \Delta \frac12|u_1|^2 &= |\nabla_1 u_1|^2 - \frac12(1-|u_1|^2)|u_1|^2 + |u_1|^2 - \ang{u,u_1}\label{norm-PDE}\,.
            \end{align}
            We apply the maximum principle for \cref{norm-PDE} and the bound $|u| \leq 3$ to deduce that $|u_1| \leq \max{\left(|u|,1\right)}\leq 3$. Then by \cref{auxilary-PDE-2} we get that:
            \begin{align*}
                \|A_1\|_{W^{2,2}(\Omega_2)} \leq C_{N,\Omega_2,\Omega_3}\,.
            \end{align*}
            Now by the Sobolev embedding $W^{2,2}(\R^2) \subset C^{\alpha}(\R^2)$ for any $0 < \alpha < 1$ and $W^{2,2}(\R^2) \subset W^{1,p}(\R^2)$ for all $1 \leq p < \infty$ we estimate that:
            \begin{align*}
                \|A_1\|_{C^\alpha(\Omega_2)} \leq C_{\alpha,N,\Omega_2,\Omega_3}\text{ and }\|A_1\|_{W^{1,p}(\Omega_2)} \leq C_{p,N,\Omega_2,\Omega_3}\,.
            \end{align*}
            Then we use this pointwise bound with \cref{auxilary-PDE-1} to see that $\Delta u_1$ is bounded in $L^2$. By standard elliptic estimates we get that:
            \begin{align*}
                \|u_1\|_{W^{2,2}(\Omega_1)} \leq  C_{N,\Omega_1,\Omega_2,\Omega_3}\Rightarrow \|u_1\|_{C^\alpha(\Omega_1)} \leq C_{\alpha,\Omega_1,\Omega_2,\Omega_3}\,.
            \end{align*}
            We use the embedding $W^{2,2}(\R^2) \subset W^{1,p}(\R^2)$ again to see that:
            \begin{align*}
                \| u_1 \|_{W^{1,p}(\Omega_1)} \leq C_{N,\Omega_1,\Omega_2,\Omega_3}\,.
            \end{align*}
            Then by \cref{auxilary-PDE-1,auxilary-PDE-2} we get that $\Delta u_1$ and $\Delta_H A_1$ are both bounded in $L^p$ for all $1<p<\infty$, so we can improve the $W^{2,2}$ estimates to $W^{2,p}$ as follows:
            \begin{align*}
                \|u_1\|_{W^{2,p}(\Omega)} + \|A_1\|_{W^{2,p}(\Omega)} \leq C_{N,\Omega , p}\,.
            \end{align*}
            From the embedding $W^{2,p}(\R^2) \subset C^{1,\alpha}(\R^2)$ for $p > 2$ and $\alpha = 1- \frac{2}{p}$ we get Holder estimates as follows:
            \begin{align}\label{apriori-estimate-1}
                \|u_1\|_{C^{1,\alpha}(\Omega)} + \|A_1\|_{C^{1,\alpha}(\Omega)} \leq C_{\alpha,N,\Omega}\,.
            \end{align}
            To gain higher regularity estimates consider the Hodge decomposition of $A$ in $U$:
            \begin{align*}
                A = d\phi + A'\,,
            \end{align*}
            Where $(u',A')$ is in the Coulomb gauge, precisely $d^*A' = 0$ in $\Omega_2$ and $A'(\nu) = 0$ on the boundary $\de\Omega_2$. Then we rewrite the system of equations \cref{auxilary-PDE-1,auxilary-PDE-2,gauge-PDE} using this decomposition:
            \begin{align}\label{auxilary-PDE-3}
            \begin{aligned}
                \Delta u_1 &= 2\ang{i du_1 , A_1} + |A_1|^2 u_1 - \frac12 (1-|u_1|^2)u_1 + (u_1 - e^{i\phi}u')\,,\\
                \Delta_H A_1 &= \ang{du_1-iu_1A_1,iu_1} + A_1-A' - d\phi\,,\\
                \Delta \phi &= \ang{u_1 -e^{i\phi}u',iu_1}\,.
                \end{aligned}
            \end{align}
            We gain higher regularity estimates with standard iteration arguments in intermediate domains starting from the apriori estaimtes \cref{apriori-estimate-1} and Schauder estimates for the system of equations \cref{auxilary-PDE-3}.
        \end{proof}
    \end{lemma}
    
    \begin{proof}[Proof of \cref{selection-principle-thm}]
        Take any section and connection $(u,\nabla)$ with energy $E(u,\nabla) = 2\pi N + \eta^2$ with small enough $\eta > 0$. First using \cref{bounding-u} we may assume $|u|\leq 3$, without loss of generality. Then we replace $(u,\nabla)$ with a minimizer of the penalized energy $\mathcal{G}_{(u,\nabla)}$ in \cref{auxilary-energy} (existence provided by \cref{auxilary-energy-existence}). In fact we repeat this process $N$ times. Then from \cref{auxilary-energy-regularity} we use the regularity estimate \textit{(i)} at the first step and \textit{(ii)} inductively after the second step. We end up with a new $N$-vortex section and connection $(\tilde{u},\tilde{\nabla})$ with the following estimates for some fixed $0<\alpha<1$:
        \begin{align*}
            \|u-\tilde{u}\|^2_{L^2(\R^2)} + \left[E(\tilde{u},\tilde\nabla) - 2\pi N\right] &\leq \eta^2\,,\\
            \left\|\tilde{u}\right\|_{C^{N,\alpha}(\Omega)} + \left\|\tilde{A} \right\|_{C^{N,\alpha}(\Omega)} &\leq C_{\Omega,N,\alpha}\,,\,.
        \end{align*}     
        for any smooth connected open domain $\Omega$ where $(\tilde{u},\tilde{A})$ is measured in the Coulomb gauge in a slightly bigger domain ($(u,\nabla)$ is also simultaneously gauge transformed). Now we consider the sublevel set $\Omega_{\frac12} = \{|\tilde u| \leq \frac12\}$ and its disjoint connected components $\cup_{k=1}^{\infty} \Omega^k_{\frac12} = \Omega_\frac12$. We start with the \textit{perturbed} vortex equations for $|\tilde u|$:
        \begin{align*}
            \Delta \log(|\tilde u|) + \frac{1-|\tilde u|^2}{2} &= \star d\left((\frac{\tilde{u}}{|\tilde{u}|})^*(d\theta)\right) + \div(\frac{e_1}{|\tilde u |}) + e_2\,,\\
            \|e_1\|^2_{L^2(\R^2)} &+ \|e_2\|^2_{L^2(\R^2)} \leq \eta^2\,.
        \end{align*}
        We test this equation with $(\frac14 - |\tilde u |^2)^+$ and estimate for each component $\Omega^k_\frac12$:
        \begin{align*}
            |\Omega^k_\frac12 \cap \{|\tilde u| \leq \frac14 \}| + \int_{\Omega^k_\frac12} |d|\tilde u||^2  \leq C\deg(\tilde u, \de \Omega^k_\frac12) + C\eta^2\,.
        \end{align*}
        First notice that for all $k\in \mathbb{N}$ the degree $\deg(\tilde u ,\Omega^k_\frac12) \geq 0$ is positive (provided $\eta$ is small enough). Then consider the set $I_0 \subset \mathbb{N}$ of indices $k$ where $\tilde u$ has zero rotation number around $\Omega^k_\frac12$ and $I\subset \mathbb{N}$ the components with positive rotation number. By mean value theorem there exists a $\frac18 \leq \beta \leq \frac14$ such that:
        \begin{align*}
            \sum_{k\in I_0} \mathcal{H}^1(\de \{|\tilde u| \leq \beta \} \cap \Omega^k_\frac12)
            \leq C\sum_{k\in I_0} \int_{\frac18}^{\frac14} \mathcal{H}^1(\de \{|\tilde u| \leq t \} \cap \Omega^k_\frac12) dt\,.
        \end{align*}
        Then by the coarea formula and then Young's inequality we get that:
        \begin{align*}
            \sum_{k\in I_0} \mathcal{H}^1(\de \{|\tilde u| \leq \beta \} \cap \Omega^k_\frac12) \leq C \sum_{k\in I_0}\int_{\{|\tilde u| \leq \beta \} \cap \Omega^k_\frac12}|d|\tilde u|| \leq C\eta^2\,,
        \end{align*}
        where we also used the measure estimates on $\Omega^k_\frac12$ for $k\in I_0$. Since for connected sets, perimeter bounds diameter, by a Vitali covering argument we find disjoint balls $B_{\rho_k}(x_k)$ for $k\in I_0$ such that:
        \begin{align*}
        \sum_{k\in I_0} \rho_k \leq C\eta^2\text{ and }
        \cup_{k\in I_0} \Omega_{\frac12}^k\cap\{|\tilde u| \leq \beta\} \subset \cup_{k\in I_0}B_{\rho_k}(x_k)\,.
        \end{align*}
        Moreover we get that $\frac18 \leq |\tilde u| \leq 3$ on $\cup_{k\in I_0} \de B_{\rho_k}(x_k)$. Then by Lipschitz bounds and diameter estimates we can \textit{clear out} the set as follows; note that $\frac18 \leq |\tilde u| \leq 3$ on the boundary of the balls with zero degree, namely $\cup_{k\in I_0} \de B_{\rho_k}(x_k)$ and $|d |\tilde u|| \leq C_N$. However since these balls have small diameter, precisely $\sum_{k\in I_0} \rho_k \leq C_N \eta^2$, we see that necessarily $|\tilde u| \geq \frac{1}{16}$ inside $\cup_{k\in I_0} B_{\rho_k}(x_k)$ (provided $\eta$ is small enough). Then for the connected components with positive rotation number we can find balls $B_{\rho_k}(x_k)$ for $k\in I$ with $|I| \leq N$ such that:
        \begin{align*}
            \max_{k\in I} \rho_k &\leq C_N\text{ and }
            \cup_{k\in I} \Omega_\frac12^k\cap\{|\tilde u| \leq \beta\} \subset \cup_{k\in I} B_{\rho_k}(x_k)\text{ and } |I| \leq N\,.
        \end{align*}
        Then for each $k \in I$ consider $(u_1,A_1)$ with uniform local $C^{N,\alpha}$ estimates. Arguing by contradiction and Arzela-Ascolli, we deduce that for small enough discrepancy $\eta >0$ the section $u_1$ is locally a $C^N$ perturbation of a solution to the vortex equations \cref{vortex-eq}:
        \begin{align*}
            \tilde u(z) = u_0(z) + R(z) \text{ for } z \in \cup_{k\in I}B_{\rho_k}(x_k)\text{ such that } \|R\|_{C^N(\cup_{k\in I}B_{\rho_k}(x_k))} \leq \epsilon_\eta\,,
        \end{align*}
        where $\epsilon_\eta$ vanishes as $\eta \rightarrow 0$. By \cref{vortex-set-estimates} and \cite{Taubes} we know that solutions are locally, up to a smooth change of gauge, complex polynomials multiplied by an analytic nonzero function. More precisely there exists an analytic function $ \Lambda_N^{-1} \leq g(z) \leq \Lambda_N$ uniformly bounded away from zero and uniform $C^N$ bounds only depending on $N$ such that:
        \begin{align*}
            \frac{\tilde u(z)}{g(z)} = \Pi_{k=1}^{M}(z-a_k) + \frac{R(z)}{g(z)} \text{ for } z \in \cup_{k\in I}B_{\rho_k}(x_k)\,.
        \end{align*}
        Finally we are in a position to apply \cref{complex-polynomial-perturbration} and conclude that $(\tilde u,\tilde A)$ satisfies the assumptions of \cref{regular-stability}.
    \end{proof}
    \subsection{Proof of stability in $\R^2$} 
    Here we prove the stability in its general form:
    \begin{proof}[Proof of \cref{main-result}] 
    By \cref{selection-principle-thm} we find $(\tilde u,\tilde A)$ satisfying the regularity assumptions of \cref{regular-stability} for some $(u_0,\nabla_0)$. Then we apply \cref{regular-stability} and estimate:
    \begin{align*}
        \|u - u_0\|_{L^2(\R^2)}^2 &\leq 2\|u - \tilde u\|_{L^2(\R^2)}^2 + 2\|\tilde u - u_0\|_{L^2(\R^2)}^2 \leq C_N\eta^2 \text{ and}\\
        \|F_\nabla - F_{\nabla_0}\|_{L^2(\R^2)}^2 &\leq 2\|F_{\nabla} - F_{\tilde\nabla}\|_{L^2(\R^2)}^2 + 2\|F_{\tilde\nabla} - F_{\nabla_0}\|_{L^2(\R^2)}^2 \leq C_N \eta^2,
    \end{align*}
    and we conclude with the proof.
    \end{proof}
    
    \subsection{The power 2 is optimal}
    One may perturb any $N$-vortex solution $(u_0,\nabla_0)$ with a smooth real valued function $h\in C_c^{\infty}(\R^2)$ to see that:
    \begin{align*}
        E(u_0e^{h} ,\nabla_0) - 2\pi N = \int_{\R^2} |u_0|^2|\star dh|^2 + |u_0|^2|\frac{e^{2h}-1}{2}|^2\,.
    \end{align*}
    We can choose $h$ to be a mollified indicator function of a ball $B_R(x)$ sufficiently far from the vortex set (as in \cref{vortex-set-estimates}) to see that:
    \begin{align*}
        E(u_0e^{h} ,\nabla_0) - 2\pi N &\geq C R^2 \geq C\|h\|_{L^2(\R^2)}^2 \\ &\geq C\|u_0e^h - u_0\|_{L^2(\R^2)}^2 \geq C\min_{(u,\nabla)\in\mathcal{F}}\|u_0e^h - u\|_{L^2(\R^2)}^2\,.
    \end{align*}
    Taking $R\to0$ we see that there exists a sequence $\{(u_k,\nabla_k)\}_{k=1}^\infty$:
    \begin{align*}
        E(u_k,\nabla_k) \to 2\pi N\text{ and } \lim_{k\to\infty}\frac{E(u_k,\nabla_k)-2\pi N}{\min_{(u_0,\nabla_0)\in\mathcal{F}}\|u_0-u_k\|_{L^2(\R^2)}^2} > 0\,.
    \end{align*}
    Hence we can see that the power $2$ may not be improved.
    
    \section{Stability for compact surfaces}
    In this section we show that the methods above can be adapted to obtain stability for nontrivial line bundles over compact Riemannian surfaces. The proofs are mostly unchanged with slight modifications. Let $(M,g)$ be a smooth Riemann surface and let $L\to M$ be a nontrivial Hermitian line bundle over $M$. Using a Stokes theorem, we have that for any section and connection $(u,\nabla)$:
    \begin{align}
        E(u,\nabla) = 2\pi |\deg(L)| + \int_{M} \left[|\nabla_{\de_1} u + i\nabla_{\de_2}u|^2 + |\star \omega - \frac{1-|u|^2}{2}|^2 \right]\,.
    \end{align}
    Naturally, the \textit{vortex equations} take the same form:
    \begin{align}
        \nabla_{\de_1} u + i\nabla_{\de_2} u = 0 \text{ and }
        \star \omega = \frac{1-|u|^2}{2}\,.
    \end{align}
    Where $\omega$ is the real two-form associated to $F_\nabla$. Integrating the second equation over $M$ we see that:
    \begin{align*}
        |\deg(L)| \leq \frac{1}{4\pi} \vol(M)\,.
    \end{align*}
    In \cite{GarcaPrada1994ADE} Garc\'ia-Prada proves if the condition above is satisfied, once we prescribe the zero set (counted with multiplicity), the solution is unique and smooth. We now recall the statement of \cref{main-result-manifold}:
    \begin{thm*}
        Let $M$ be a smooth compact Riemannian surface and $L\to M$ a Hermitian line bundle over $M$ with $0\leq \deg(L) \leq \frac{1}{4\pi} \vol(M)$. Then there exists a constant $C_{M}>0$ depending only on $M$, with the following property: Let $(u,\nabla)\in W^{1,2}(M)$ be a section and connection on $L$ such that $E(u,\nabla) - 2\pi\deg(L)$ is small enough, then:
        \begin{align*}
            \min_{(u_0,\nabla_0)\in \mathcal{F}} \|u-u_0\|_{L^(M)}^2 + \|F_\nabla - F_{\nabla_0}\|_{L^(M)}^2 \leq C_{M}\left[E(u,\nabla) - 2\pi\deg(L)\right]\,,
        \end{align*}
        where $\mathcal{F}$ is the family all minimizers of the Yang-Mills-Higgs energy on $L\to M$.
    \end{thm*}
    \begin{rmk*}
    Before diving into the proof, note that on compact surfaces the Sobolev space $W^{1,2}(M)$ embeds into the space of functions of vanishing mean oscillation $\mathrm{VMO}(M)$, hence the degree is well defined.
    \end{rmk*}
        \subsection{Preliminary estimates on solutions}
        Here we prove some facts about solutions of the vortex equations on smooth compact surfaces. First using the vortex equations, we can see that $$\frac12 \Delta |u_0|^2 = |\nabla_0 u_0|^2 - \frac{1}{2}|u_0|^2(1-|u_0|^2)$$ and by an application of the maximum principle, since $M$ is compact, we can see that $|u|\leq 1$ on $M$.
        \begin{proposition}\label{estiamtes-on-solutions-manifold-prop}
            For any compact smooth Riemann surface $(M,g)$ there exists small constants $c_M,\beta_M>0$ and a constant $C_M>1$ depending on $M$ with the following property. Let $L\to M$ be a Hermitian line bundle on $M$ with $0\leq \deg(L) \leq \frac{1}{4\pi}\vol(M)$ and $(u_0,\nabla_0)$ be a solution to the vortex equations \cref{vortex-eq} with the prescribed zero set $x_1,\dots,x_{\deg(L)} \in M$ counted with multiplicity. Then there exists a geodesic ball $B_{\rho}(x_0)$ such that:
            \begin{enumerate}[label=(\roman*)]
                    \item $|u_0| \geq \beta_M>0$ on $B_{2\rho}(x_0)$ with $c_M \leq \rho \leq 2 c_M$\,,
                    \item $C_M^{-1} \omega \leq |u_0| \leq C_M \omega$, where $\omega(x) = \Pi_{k=1}^{\deg(L)} d(x,x_k)$\,,
                    \item All geodesic balls of radius less than $2c_M$ are uniformly bi-Lipschitz to euclidean balls of comparable radius.
            \end{enumerate}
            Here $d(x,y)$ is the geodesic distance between $x,y$ on $M$.
            
            \begin{proof}
                The proof is essentially the same as \cref{vortex-set-estimates} with some modifications in the case of compact surfaces.
                First notice that the modulus $r_0=|u_0|$ satisfies a similar equation:
                \begin{align}\label{vortex-pde-manifold}
                    -\Delta_g \log(r_0) + \frac12(r_0^2-1) = -\sum_{k=1}^{\deg(L)} 2\pi\delta_{x_k}\,.
                \end{align}
                Multiplying the above display by $(\beta^2-r_0^2)^{+}$ and integrating by parts we see that:
                \begin{align*}
                    \int_{\{r_0 \leq \beta\}}|dr_0|^2 + \frac12(1-r_0^2)(\beta^2-r_0^2) = 2\pi\beta^2\deg(L)\,.
                \end{align*}
                Now we take $\beta_M$ small enough and using the smooth coarea formula in the place of the euclidean one, following the proof of \textit{(i)} in \cref{vortex-set-estimates} we can cover the vortex set $\{r_0 \leq \beta_M\}$ with a collection of $n\leq \deg(L)$ geodesic balls $\{B_{\sigma}(z_k)\}_{k=1}^n$ with small enough radius $\sigma>0$; namely, $\{r_0 \leq \beta_M\} \subset \bigcup_{k=1}^{n} B_{\sigma}(z_k)$. Now since $n\leq \deg(L) \leq \frac{1}{4\pi}\vol(M)$, we can take $\sigma,\beta_M$ small enough such that the complement contains a geodesic ball. Precisely, there exists a point $x_0$ and a radius $\rho>c_M$ such that $r_0> \beta_M$ on $B_{2\rho}(x_0)$. This proves \textit{(i)}.
                
                Now we prove part \textit{(ii)} in each ball $B_\sigma(z_k)$. Consider the zeros of $r_0$ inside $B_{\sigma}(z_k)$ and name them $x_1,\dots,x_{n_k}$ (without loss of generality). Now consider the function $\tilde\omega_k(x)=\Pi_{j=1}^{n_k} e^{-G_{x_j}(x)}$, where $G_p(x)$ is the Green's function for the ball $B_{2\sigma}(z_k)$ centered on $p$. We see that:
                \begin{align*}
                    -\Delta_g \log(\tilde\omega) = -\sum_{j=1}^{n_k}\delta_{x_j}\text{ inside } B_{\sigma}(z_k)\,.
                \end{align*}
                Subtracting the above display from \cref{vortex-pde-manifold} and using the maximum principle, we can see that $\|\log(\frac{\tilde\omega_k}{r_0})\|_{L^\infty(B_\sigma(z_k))} \leq C_M$. By \cite[eq (1.1)]{Li-Tam} (see \cref{weights-on-manifolds} and the paragraph after) we can see that $\tilde\omega_k$ is locally comparable to $\Pi_{j=1}^{n_k}d(x,x_j)$. Then we put together this estimate for all $k=1,\dots,n$ and use the global bound $|u_0|\leq 1$ together with compactness of $M$ to obtain \textit{(ii)}.

                Item \textit{(iii)} simply follows by compactness and choosing $c_M$ small enough.
            \end{proof}
        \end{proposition}
        \subsection{Stability for regular pairs}
        The goal now is to prove the stability for regular enough pairs (analogous to \cref{regular-stability}), in the following theorem:
        \begin{thm}\label{regular-stability-manifold}
            For any compact smooth Riemann surface $(M,g)$ and $\Lambda \geq 1$ there exists $\eta_{\Lambda,M},C_{\Lambda,M} > 0$ with the following property: Let $L\to M$ be a Hermitian line bundle on $M$ with $0\leq \deg(L) \leq \frac{1}{4\pi}\vol(M)$ and let $(u,\nabla)$ be a pair such that:
            \begin{enumerate}[label=(\roman*)]
                    \item $\star d((\frac{u}{|u|})^* (d\theta)) = 2\pi \sum_{k=1}^{\deg(L)}\delta_{x_k}$,
                    \item $E(u,\nabla) - 2\pi \deg(L) \leq \eta_{\Lambda,M}^2$,
                    \item $\Lambda^{-1} |u_0| \leq |u| \leq \Lambda |u_0|$,
            \end{enumerate}
            where $(u_0,\nabla_0)$ is the solution of the vortex equations \cref{vortex-eq} on $L\to M$ with the zero set $x_1,\dots,x_{\deg(L)} \in M$ (counted with multiplicity). Then:
            \begin{align*}
                \||u|-|u_0|\|^2_{L^2(M)} + \|F_\nabla - F_{\nabla_0} \|^2_{L^2(M)} \leq C_{\Lambda,M} \left[E(u,\nabla) - 2\pi N\right]\,.
            \end{align*}
        \end{thm}
        
        First we prove the analogous statement to \cref{compactness-inequality-1}. We need to subtract a geodesic ball from the manifold, since the estimates of \cref{inequalities-appendic-section} only work on surfaces with boundary. More precisely, there are no non-constant weights that satisfy \cref{weight-weak-condition} on compact surfaces, so the statements of \cref{inequalities-appendic-section} are empty on a compact manifold.
        \begin{proposition}\label{compactness-inequality-manifold-prop}
            Let $(M,g)$ be a Riemannian surface. Then for any $\Lambda>1$ and integer $n>0$ there exists a constant $C_{n,\Lambda,M}$ with the following property. Let $\omega$ be a weight as in \cref{weights-on-manifolds} with integer powers. Precisely:
            \begin{align*}
                \omega = \Pi_{k=1}^{n} d(x,x_k)\,\text{ for a collection } \{x_k\}_{k=1}^n \subset M \text{ counted with multiplicity}\,.
            \end{align*}
            Moreover let $B_{\rho}(x_0)$ be a geodesic ball with radius $c_M \leq \rho \leq 2c_{M}$. Then for any compactly supported function $h \in C^{\infty}_c(M\backslash B_{\rho}(x_0))$ and one-form $B\in C^\infty_c(\bigwedge^1 M\backslash B_{\rho}(x_0))$ the following weighted inequality holds:
            \begin{align*}
                \int_{M\backslash B_{\rho}(x_0)} \omega^2|h|^2  \leq C_{n,\Lambda,M} \int_{M\backslash B_\rho(x_0)} \left[\omega^2|\star dh + B|^2 + |\star dB + V(x) h|^2\right]\,,
            \end{align*}
            provided that $0 \leq V(x) \leq \Lambda \omega(x)^{1+\frac{1}{n}}$.
            \begin{proof}
                Since the estimates in \cref{inequalities-appendic-section} have universal constants and work on arbitrary surfaces with boundary, we can apply the proof of \cref{compactness-inequality-1} almost verbatim. The only difference is that we also need to keep track of a sequence of geodesic balls $B_{\rho_k}(x_{0^k})$ subtracted from the manifold. However since the radius is bounded above and below away from zero ($c_M \leq \rho_k \leq 2c_M$) and the manifold $M$ is compact, we can extract a sub-sequential limit to some $M\backslash B_{\rho_{\infty}}(x_{0^\infty})$ with $c_M\leq \rho_{0^\infty} \leq 2c_M$. The rest of the proof is unchanged.
            \end{proof}
        \end{proposition}
        
        In the rest of this subsection we work to patch the estimates of the subtracted ball to the rest of the manifold and conclude the stability for regular enough pairs.
        \begin{proof}[Proof of \cref{regular-stability-manifold}]
            Up to conjugating $u$, assume that $\deg(L)\geq0$. Now observe that, after identifying $\nabla: d-iA$ and $u=r e^{i\theta}$ for some $\theta:M\to S^1$ and one-form $A\in \bigwedge^1(M)$ the energy has the following form:
            \begin{align*}
                E(u,\nabla) - 2\pi\deg(L) = \int_{M} \left[ r^2|\star d\log(r) + A-d\theta|^2 + |\star dA - \frac{1-r^2}{2}|^2 \right] \,.
            \end{align*}
            Then by the assumption \textit{(i)} we can gauge fix such that $\frac{u}{|u|}=\frac{u_0}{|u_0|}$, namely $u,u_0$ have equal phase. Then naming $h=\log(\frac{r}{r_0})$ for $r_0=|u_0|$ and $B = A-A_0$ for $\nabla_0:d-iA_0$, we see that:
            \begin{align*}
                E(u,\nabla) - 2\pi\deg(L) =  \int_{M} \left[r^2|\star dh + B|^2 + |\star dB + V(x) h|^2 \right]\,,
            \end{align*}
            where $V(x) = r_0^2\frac{e^{2h}-1}{2h}$. By assumption \textit{(iii)} we can see that $\|h\|_{L^\infty(M)} \leq \log(\Lambda)$, hence $ C_{\Lambda}^{-1} r_0^2 \leq V(x) \leq C_{\Lambda} r_0^2$. Now we proceed similar to the proof of \cref{regular-stability}. In fact, we prove that as a consequence of \cref{compactness-inequality-manifold-prop} and the Poincar\'e inequality on small geodesic balls, the discrepancy cannot concentrate on the subtracted ball. In fact we aim to prove that there exists some constant $C_M$ such that:
            \begin{align*}
                \int_{M} r^2|h|^2 \leq C_M \int_{M} \left[r^2|\star dh + B |^2 + |\star dB + V(x) h|^2\right]\,.
            \end{align*}
            Arguing by contradiction, assume there is a sequence $\{r_k,h_k,B_k,V_k\}_{k=1}^\infty$ satisfying the assumptions, such that:
            \begin{align*}
                &\int_{M} r_k^2|h_k|^2 = 1\,,\\
                \eta_k^2= &\int_{M} \left[r_k^2|\star dh_k + B_k |^2 + |\star dB_k + V_k(x)h_k|^2\right] \to 0 \,.
            \end{align*}
            Defining $e_1 = \star dh_k + B_k$ and $e_2 = \star dB_k + V_k(x)h_k$ the above display takes the following form:
            \begin{align}\label{pde-manifold-error}
                -\Delta_g h_k + V_k(x)h_k = \star d(\frac{e_1}{r_k}) + e_2 \text{ with } \|e_1\|_{L^2(M)}^2+\|e_2\|_{L^2(M)}^2 \leq C\eta_k^2\,.
            \end{align}
            Now take the geodesic ball $B_{\rho_k}(x_k)$ as in \textit{(i)} \cref{estiamtes-on-solutions-manifold-prop} and define the two cut-off functions $0\leq \psi_k,\phi_k \leq 1$ as follows:
            \begin{align*}
                \begin{cases}
                    \phi_k = 1 \text{ on } B_{\rho_k/2}(x_k)\,,\\
                    \phi_k = 0 \text{ on } M\backslash B_{\rho_k}(x_k)\,,\\
                    |d\phi_k| \leq C_M\,.
                \end{cases}
                \begin{cases}
                    \psi_k = 1 \text{ on } M\backslash B_{\rho_k/2}(x_k)\,,\\
                    \psi_k = 0 \text{ on } B_{\rho_k/3}(x_k)\,,\\
                    |d\psi_k| \leq C_M\,.
                \end{cases}
            \end{align*}
            First we test \cref{pde-manifold-error} with $\phi_k h_k$, integrate by parts and use the estimates for $r_k$ on $B_{\rho_k}(x_k)$ with Young's inequality to see that:
            \begin{align}\label{temp-estimate-manifold-1}
                \int_{B_{\rho_k/2}(x_k)} |h_k|^2 + |dh_k|^2 \leq C_{M,\Lambda} \eta_k^2 + \int_{B_{\rho_k}(x_k)} |h_k|^2\,.
            \end{align}
            Now we apply \cref{compactness-inequality-manifold-prop} for $\psi_kh_k,\psi_k B_k$ to see that (using assumption  \textit{(iii)} and \textit{(iii)} in \cref{estiamtes-on-solutions-manifold-prop}):
            \begin{align*}
                &\int_{M\backslash B_{\rho_k/2}(x_k)} r_k^2|h_k|^2 \leq \int_{M\backslash B_{\rho_k/3}(x_k)} r_k^2 |\psi_k h_k|^2 \\ 
                &\leq C_{M,\Lambda}\int_{M\backslash B_{\rho_k/3}(x_k)} \left[r_k^2|\star d (\psi_k h_k) + \psi_k B_k|^2 + |\star d (\psi_k B_k) + V_k(x)(\psi_kh_k)|^2 \right] \\
                &\leq C_{M,\Lambda}\left[\eta_k^2 + \int_{B_{\rho_k/2}(x_k)} |h_k|^2 + |dh_k|^2  \right] \leq_{\cref{temp-estimate-manifold-1}} C_{M,\Lambda}\left[\eta_k^2 + \int_{B_{\rho_k}(x_k)} |h_k|^2\right] \,.
            \end{align*}
            Notice that on $B_{\rho_k}(x_k)$ we have $r_k \geq \beta_M>0$. Combining the above display with $\|r_kh_k\|_{L^2(M)}=1$ we see that for large enough $k$:
            \begin{align*}
                \int_{B_{\rho_k}(x_k)} |h_k|^2  \geq C_{M,\Lambda} >0\,.
            \end{align*}
            Similarly testing \cref{pde-manifold-error} with $\chi_kh_k$ such that $\chi_k = 1 $ on $B_{\rho_k}(x_k)$ and $\chi_k=0$ on $M\backslash B_{2\rho_{k}}(x_k)$, we see that:
            \begin{align*}
                \int_{B_{\rho_k}(x_k)} |h_k| + |dh_k|^2  \leq C_{M,\Lambda}\,.
            \end{align*}
            Note that by \textit{(iv)} in \cref{estiamtes-on-solutions-manifold-prop} and $c_M \leq \rho_k \leq 2c_M$ the geodesic balls $B_{\rho_k}(x_k)$ are uniformly bi-Lipschitz to the unit disk in the euclidean plane. Combining the last two displays, using Banach-Alagolu, Rellich–Kondrachov theorem together we the bounds on the radius, we may extract a convergent sub-sequence in $W^{1,2}$. Moreover the weak limit satisfies \cref{pde-manifold-error} with $\|e_1\|_{L^2(M)}=\|e_2\|_{L^2(M)}=0$, testing the equation by $h$ we see that any weak limit of $h_k$ should be $0$. However since the $L^2$ norm of $h_k$ is uniformly bounded below and away from zero, by the strong $L^2$ convergence, contradiction follows.
        \end{proof}
        Similar to \cref{compactness-inequality-2} we have that:
        \begin{corollary}\label{compactness-inequality-2-manifold}
            As a consequence of the estimates above we also get that there exists a constant $C_{\Lambda,M}>0$ such that for any $0<\epsilon<\frac1N$:
            \begin{align*}
                \int_{M} \omega^{2+2\epsilon} |dh|^2 \leq \frac{C_{\Lambda,R,N}}{\epsilon^2} \int_{M} \omega^2|\star dh + B|^2 + |\star dB + V(x)h|^2\,.
            \end{align*}
        \end{corollary}
        \subsection{Selection principle and proof of stability on compact surfaces}
        Here we find a \textit{regular enough} pair satisfying the assumptions of \cref{regular-stability-manifold} second order close to any nearly minimizing pair $(u,\nabla)\in W^{1,2}(M)$. The proof is adapted with little modification.
        \begin{thm}\label{selection-principle-manifold}
            Let $(M,g)$ be a compact Riemannian surface. Then there exists constants $\eta_M,C_M,\Lambda_M$ with the following property: Let $L\to M$ be a Hermitian line bundle with $0\leq \deg(L) \leq \frac{1}{4\pi}\vol(M)$. Then for any pair of section and connection $(u,\nabla)\in W^{1,2}(M)$ such that $E(u,\nabla) - 2\pi\deg(L) \leq \eta_M^2$. Then there exists another pair $(\tilde{u},\tilde\nabla)$ with the following properties:
            \begin{enumerate}[label=(\roman*)]
                \item $\|u-\tilde{u}\|^2_{L^2(M)} + \|F_\nabla-F_{\tilde\nabla}\|^2_{L^2(M)} \leq C_N \left[E(u,\nabla) - 2\pi\deg(L)\right]$\,.
                \item $\Lambda_M^{-1} |u_0| \leq |\tilde{u}| \leq \Lambda_M |u_0|$ for some solution $(u_0,\nabla_0)$ with $E(u_0,\nabla_0) = 2\pi\deg(L)$\,.
                \item $2\pi\deg(L) \leq E(\tilde{u},\tilde\nabla) \leq E(u,\nabla)$\,.
            \end{enumerate}
            \begin{proof}
                The proof is essentially the same as \cref{selection-principle-thm}. Similar to \cref{bounding-u} we can assume without loss of generality that $|u|\leq 3$. Then we replace $(u,\nabla)$ with the minimizer of the following auxiliary energy:
                \begin{align*}
                    \G_{(u,\nabla)}(u_1,\nabla_1) = E(u_1,\nabla_1) + \|u-u_1\|^2_{L^2(M)}+\|A-A_1\|^2_{L^2(M)}\,.
                \end{align*}
                The existence and regularity follow verbatim as in \cref{auxilary-energy-existence} and \cref{auxilary-energy-regularity} respectively. (In fact the proof of existence is a bit simplified since it is straightforward to prove that the degree passes to the limit) Iterating the process $\deg(L)$ times, we see that we can find a pair $(\tilde{u},\tilde\nabla)$ such that:
                \begin{align*}
                    \|\tilde{u}-u\|_{L^2(M)}^2 + \|F_{\tilde\nabla}-F_\nabla\|_{L^2(M)}^2 \leq C_{M}\left[E(u,\nabla)-2\pi\deg(L)\right]\,.
                \end{align*}
                Moreover, for any simply connected domain $\Omega\in M$ there exists a constant $C_{\Omega,M}$ with the following property: After a suitable gauge transformation $(\tilde{u},\tilde{A}) \to (\tilde{u}e^{i\xi},\tilde{A}+d\xi)$ for $\xi \in W^{1,2}_0(\Omega)$ such that:
                \begin{align*}
                    \Delta_g \xi = d^*\tilde{A} \text{ inside } \Omega\subset M \text{ and } \nabla_\nu \xi = \iota_\nu \tilde{A} \text{ on } \de\Omega\,,
                \end{align*}
                we have the following estimate:
                \begin{align*}
                    \|\tilde{u}e^{i\xi}\|_{C^{\deg(L),\alpha}(\Omega)} \leq C_{M,\Omega}\,.
                \end{align*}
                Arguing by contradiction and compactness, using Arzala-Ascoli we can see that if $\eta^2_M$ is small enough, then $\tilde{u}e^{i\xi}$ is a $C^N$ perturbation of some solution $u_0$. Now we take a local chart of $U\subset\Omega$ and map everything onto an euclidean ball; then taking small enough domains and arguing by compactness, we can apply \cref{complex-polynomial-perturbration} to conclude as in the proof of \cref{selection-principle-thm}.
            \end{proof}
        \end{thm}
        \begin{proof}[Proof of \cref{main-result-manifold}]
            Using \cref{selection-principle-manifold} and \cref{regular-stability-manifold}, the proof of \cref{main-result} applies verbatim.
        \end{proof}
        \begin{proof}[Proof of \cref{main-result-epsilon-sobolev-manifold}]
            The proof follows from \cref{compactness-inequality-2-manifold}, similar to the proof of \cref{main-result-epsilon-sobolev} and \cref{jacobian-main-result}.
        \end{proof}
    \appendix
    \section{Smooth perturbation of complex polynomials}
    \begin{lemma}\label{complex-polynomial-perturbration}
        For any integer $N>0$ there exists constants $\Lambda_N>1$ and $\epsilon_N>0$ with the following property: Let $P(z) = \Pi_{k=1}^{N}(z-a_i)$ be a complex polynomial with degree $N$ with $a_1,\dots,a_N \in B_{\frac12}(0)$. Then for any perturbation $R: B_1(0)\rightarrow \C$ with $\|R\|_{C^{N}(B_1(0))} \leq \epsilon_N$ there exists another complex polynomial $Q(z) = \Pi_{k=1}^{N}(z-b_j)$ with $b_1,\dots,b_N \in B_{\frac23}(0)$ such that:
        \begin{align*}
            \Lambda_N^{-1}\leq \frac{|P(z) + R(z)|}{|Q(z)|} \leq \Lambda_N\,.
        \end{align*}
        \begin{proof}
            We prove the lemma by induction. By the bound $|R(z)| \leq \epsilon_N \leq |P(z)|$ on $\de B_1$ and Rouche's theorem, there exists a point $a \in B_1$ such that $P(a) + R(a) = 0$. Then we define $\tilde{R}(z) = R(z) + P(a)$ so that $\tilde R(a) = 0$ and we write:
            \begin{align*}
                \frac{P(z) + R(z)}{z-a} = \frac{P(z) - P(a) + \tilde{R}(z)}{z-a} = \frac{P(z)-P(a)}{z-a} + \frac{\tilde{R}(z)}{z-a}\,.
            \end{align*}
            Since $\tilde{R}(z) = R(z) - R(a)$ we get that $\|\tilde{R}\|_{C^N(B_1)} \leq 2\epsilon_N$. Then by a Taylor expansion we estimate:
            \begin{align*}
                \|\frac{\tilde{R}(z)}{z-a}\|_{C^{N-1}(B_1)} \leq C_N \|\tilde{R}\|_{C^N(B_1)} \leq C_N \epsilon_N\,.
            \end{align*}
            Now since $\frac{P(z)-P(a)}{z-a}$ is a complex polynomial with degree $N-1$, by induction we see that there exists a complex polynomial $\tilde{Q}(z)$ such that:
            \begin{align*}
                \Lambda_{N-1}^{-1} \leq  \big| \frac{P(z)-P(a)+\tilde{R}(z)}{\tilde{Q}(z)(z-a)} \big| \leq \Lambda_{N-1}\,.
            \end{align*}
            Naming $Q(z) = \tilde{Q}(z)(z-a)$ we get that:
            \begin{align*}
                \Lambda_N^{-1} \leq \frac{|P(z)+R(z)|}{|Q(z)|} \leq \Lambda_N\,.
            \end{align*}
            The case $N=1$ follows by a standard transversality argument.
        \end{proof}
    \end{lemma}
    \section{Weighted elliptic inequalities}\label{inequalities-appendic-section}
    Here we state some results from \cite{Halavati-inequality} for convenience of the reader. Let $\Omega$ be a smooth bounded open domain $\Omega\subset M$ of a smooth Riemannian two-manifold. We work with weights $\omega\in W^{1,2}(\Omega)$ that formally satisfy:
    \begin{align*}
        \omega^2\Delta_g \log(\omega) = 0 \text{ in } \Omega,
    \end{align*}
    in a weak sense. Precisely for any $\phi\in C^{\infty}_c(\Omega)$ the following identity holds:
    \begin{align}\label{weight-weak-condition}
          \int_{\Omega} \left[4|d\omega|^2\phi - \omega^2 \Delta_g\phi \right] = 0\,.
    \end{align}
    In the case of $M:=\R^2$ all weights of the form
    \begin{align*}
        \omega(x) = \Pi_{k=1}^n |x-a_k|^{\alpha_k} \text{ with } \{a_k\}_{k=1}^{n} \subset \R^2 \text{ and } \alpha_k > 0\,,
    \end{align*}
    are admissible. More generally (as in \cite{Halavati-inequality}) for a smooth open and bounded domain $\Omega\subset M$ in a smooth two-manifold, the weights can take the following form:
    \begin{align}\label{weights-on-manifolds}
        \omega(x) = \Pi_{k=1}^n e^{-\alpha_kG_{p_k}(x)} \text{ with } \{p_k\}_{k=1}^n \subset \Omega \text{ and } \alpha_k > 0\,,
    \end{align}
    where $G_{p}(x) = G(p,x)$ is the Green's function for the domain $\Omega$ centered on $p$, namely the fundamental solution for the Laplacian on $\Omega$ (for a comprehensive account of the Green's function on smooth manifolds see \cite{Li-Tam}). Following the observation in \cite[eq (1.1)]{Li-Tam}, we see that there is some constant $C>0$ such that any weight of the form \cref{weights-on-manifolds} satisfies:
    \begin{align*}
        C^{-n}\omega(x) \leq \Pi_{k=1}^{n} d(x,p_k)^{\alpha_k}\leq C^n\omega(x)\text{ in } \Omega \,,
    \end{align*}
    where $d(x,y)$ is the geodesic distance between $x,y$ on $M$. Now we state the theorems:
    \begin{thm}\label{Generalized-CKN-thm}
      Let $(M^2,g)$ be a smooth $2$-manifold and $\Omega\subset M^2$ a smooth open domain with boundary. Consider a weight $\omega$ as in \cref{weight-weak-condition}. Then for any compactly supported function $f\in C^{\infty}_c(\Omega)$ we have that:
      \begin{align*}
        \int_{\Omega} |\nabla\omega|^2 |f|^2 \leq \int_{\Omega} \omega^2 |\nabla f|^2\,.
      \end{align*}
    \end{thm}
    The next proposition contains the ideas of the \textit{weighted Hodge decomposition} in a general settings:
    \begin{proposition}\label{hodge-prop}
        Let $B$ be a smooth compactly supported one form $B\in C^\infty_{c}(\bigwedge^1 \Omega)$ and $\Omega$ a smooth simply connected open domain with boundary on a two-manifold $\Omega \subset M^2$ and let $\omega$ be a weight satisfying \cref{weight-weak-condition}. Then we have the following:
        \begin{enumerate}[label=(\roman*)]
            \item There exists functions $h,f$ with $\int_{\Omega}\omega^2|dh|^2 + \omega^{-2}|df|^2 < \infty $ such that $h$ minimizes the weighted functional $\int_{\Omega} \omega^2|B - \star dh|^2$. Consequently, we have the \textit{weighted Hodge decomposition} as follows:
            \begin{align*}
                \omega B = \star \omega dh + \omega^{-1} df\,.
            \end{align*}
            \item Let $\psi,\phi \in W^{1.2}_0(\Omega)$ be the exact and co-exact part of the standard Hodge decomposition of $B$
            \begin{align*}
                B = \star d\phi + d\psi\,.
            \end{align*}
            Then we have the estimates below for any $\epsilon > 0$:
            \begin{align*}
                \|\omega^{1+\epsilon} (dh - d\phi)\|^2_{L^2(\Omega)} \leq C\frac{\sup_{\Omega} \omega^{2\epsilon}}{\epsilon^2}\|\omega^{-1} df\|^2_{L^2(\Omega)}\,.
            \end{align*}
            \item Let $\omega_k$,$B_k$ be a sequence of weights and one-forms such that
            \begin{align*}
                \omega_k B_k = \star \omega_k dh_k + \omega_k^{-1}df_k& \text{ and } B_k = \star d\phi_k + d\psi_k \\ \text{ with } \omega_k^{-1} df_k \rightarrow 0 &\text{ strongly in }L^2(\Omega)\,.
            \end{align*}
            Then for any fixed $\epsilon > 0$, we have that:
            \begin{align*}
                \omega_k^{1+\epsilon} (h_k - \phi_k) \rightarrow 0 \text{ strongly in } L^2(\Omega) \text{ and weakly in } W^{1,2}(\Omega)\,.
            \end{align*}
        \end{enumerate}
        
        \begin{proof}
            Define the weighted Sobolev space $W^{1,2}_0(\omega^2,\Omega)$ as the completion of $C_c^{\infty}(\Omega)$ under the weighted norm $\int_{\Omega}\omega^2(|du|^2 + |u|^2)$. This norm is equivalent to $\int_{\Omega} |d(\omega u)|^2 $ by \cref{Generalized-CKN-thm} and Poincar\'e inequality:
            \begin{align*}
                \int_\Omega |d\omega|^2 |u|^2 \leq \int_{\Omega} \omega^2|du|^2\text{ and }
                \int_{\Omega} \omega^2 |u|^2 \leq C_\Omega \int_{\Omega}|d(\omega u)|^2 \leq C_\Omega \int_{\Omega}\omega^2 |du|^2\,.
            \end{align*}
            We can guarantee the existence of a minimizer of $\int_{\Omega}\omega^2|B - \star dh|^2$ by the direct method of the calculus of variations. The Euler Lagrange equations for minimizers tell us that:
            \begin{align*}
                d(\omega^2(B - \star dh)) = 0 \Rightarrow  &\text{ there exists } f \in W^{1,2}_0(\omega^{-2},\Omega) \text{ such that } \omega^2(B - \star dh) = df\,.
            \end{align*}
            We also write the standard Hodge decomposition of $B$
            \begin{align*}
                B = \star d\phi + d\psi\,.
            \end{align*}
            Then observe that:
            \begin{align*}
                \star d\phi + d\psi =  \star dh + \omega^{-2} df \Rightarrow \Delta_g (h - \phi) = 2\frac{d\omega}{\omega^3}\wedge df\,.
            \end{align*}
            Then we can estimate
            \begin{align*}
                \int_{\Omega} \frac{\omega^4}{|d\omega|^2} |\Delta_g(h-\phi)|^2 \leq 4\int_{\Omega} \omega^{-2}|df|^2\,.
            \end{align*}
            Now we apply \cref{Generalized-CKN-thm} directly to see
            \begin{align*}
                \int_{\Omega} \omega^{2+2\epsilon}|d(h-\phi)|^2 \leq C \frac{(\sup_{\Omega} \omega)^{2\epsilon}}{\epsilon^2} \int_{\Omega} \frac{\omega^4}{|d\omega|^2} |\Delta_g(h-\phi)|^2\,,
            \end{align*}
            and we finally estimate
            \begin{align*}
                \int_{\Omega} \omega^{2+2\epsilon}|d(h-\phi)|^2 \leq  C \frac{(\sup_{\Omega} \omega)^{2\epsilon}}{\epsilon^2}\int_{\Omega} \omega^{-2}|df|^2\,.
            \end{align*}
            Where $C \leq \frac{8\epsilon^2 + 5(1+\epsilon)^4}{8(1+\epsilon)^2} $ and is comparable to $\frac58$ as $\epsilon\rightarrow 0$. This also proves \textit{(iii)}.
        \end{proof}
    \end{proposition}
    \begin{ack*}
      I would like to thank Guido De Philippis, Fang-Hua Lin and Alessandro Pigati for their support and mentorship and Robert V.\ Kohn for their interest and related discussions. I also want to thank Zheng-Jiang Lin for discussions regarding \cref{complex-polynomial-perturbration}. I would also like to thank the anonymous referees for their careful reading and valuable suggestions.
      
      The author has been supported by the NSF grant DMS-2055686 and the Simon's foundation.
    
      \textit{All I am and will ever be is owed to you; Sooshiant, Sedigheh and Naser.}
    \end{ack*}
    \printbibliography

\end{document}